\documentclass[11pt, reqno]{amsart}

\usepackage{amssymb}
\usepackage[shortlabels]{enumitem}
\usepackage{etoolbox}
\usepackage{eucal}
\usepackage[margin=1in]{geometry}
\usepackage[colorlinks=true, allcolors=blue]{hyperref}
\usepackage[dvipsnames]{xcolor}
\usepackage[all]{xy}

\numberwithin{equation}{section}
\allowdisplaybreaks

\newtoggle{comments}
\newtoggle{details}
\newtoggle{detailsnote}

\toggletrue{comments}      % For comments
\toggletrue{details}       % For details
\toggletrue{detailsnote}   % For including note about details toggle
\newcommand\C{\mathbb{C}}
\newcommand\Z{\mathbb{Z}}

\newcommand\N{\mathbb{N}}
\newcommand {\Zp} {\mathbb{Z}_{> 0}}

\newcommand\bbT{\mathbb{T}}
\newcommand\bbTsl{\mathbb{T}_{\mathfrak{sl}(\infty)}}
\newcommand\bbTgl{\mathbb{T}_{\mathfrak{gl}(\infty)}}

\newcommand\tbbTsl{\widetilde{\mathbb{T}}_{\mathfrak{sl}(\infty)}}
\newcommand\tbbTgl{\widetilde{\mathbb{T}}_{\mathfrak{gl}(\infty)}}

\newcommand\Mod{\operatorname{-mod}}

\newcommand\g{\mathfrak{g}}

\newcommand\fs{\mathfrak{s}}

\newcommand\h{\mathfrak{h}}
\newcommand\n{\mathfrak{n}}
\newcommand\fa{\mathfrak{a}}
\newcommand\fo{\mathfrak{o}}
\newcommand\fsp{\mathfrak{sp}}
\newcommand\fosp{\mathfrak{osp}}

\newcommand\fb{\mathfrak{b}}
\newcommand\fsl{\mathfrak{sl}}
\newcommand\fgl{\mathfrak{gl}}
\newcommand\fso{\mathfrak{so}}

\newcommand\ft{\mathfrak{t}}

\newcommand\fk{\mathfrak{k}}

\newcommand\cE{\mathcal{E}}

\newcommand\cT{\mathcal{T}}

\newcommand\bU{\mathbf{U}}

\DeclareMathOperator{\im}{im} %Image of a map
\DeclareMathOperator{\Hom}{Hom}

\DeclareMathOperator{\Ext}{Ext}

\DeclareMathOperator{\End}{End}

\DeclareMathOperator{\Int}{Int}

\DeclareMathOperator{\Span}{Span}

\DeclareMathOperator{\Der}{Der}

\DeclareMathOperator{\Ann}{Ann}
\DeclareMathOperator{\Supp}{Supp} % Support

\DeclareMathOperator{\Ind}{Ind}
\DeclareMathOperator{\Coind}{Coind}

\DeclareMathOperator{\ad}{ad}

\DeclareMathOperator{\soc}{soc}

\theoremstyle{plain}
\newtheorem{theo}{Theorem}[section]
\newtheorem*{theo*}{Theorem}
\newtheorem{prop}[theo]{Proposition}
\newtheorem{lem}[theo]{Lemma}
\newtheorem{cor}[theo]{Corollary}

\theoremstyle{definition}
\newtheorem{defin}[theo]{Definition}
\newtheorem*{rem*}{Remark}
\newtheorem{rem}[theo]{Remark}

\newtheorem{example}[theo]{Example}

\numberwithin{equation}{section}

\allowdisplaybreaks

\toggletrue{detailsnote}   % For including note about details toggle

\iftoggle{comments}{%
  \newcommand{\comments}[1]{
    \begin{center}
      \parbox{6.5 in}{
        \color{red}
          {\footnotesize \textbf{Comments:} #1}
        \color{black}}
    \end{center}}
}{%
  \newcommand{\comments}[1]{}
}

\iftoggle{details}{%
  \newcommand{\details}[1]{
      \ \\
      \color{OliveGreen}
        \begin{footnotesize}
          \textbf{Details:} #1
        \end{footnotesize}
      \color{black}
      \\
  }
}{%
  \newcommand{\details}[1]{}
}

\begin{document}
%
%%%%%%%%%%%%%%%%%%%%%%%%%%%%%%%%%%%

\title[Representations of $W(\infty)$]{Representations of the Lie superalgebra of superderivations of the Grassmann algebra at infinity}
\author{Lucas Calixto}
\address{L.~Calixto: Department of Mathematics, Federal university of Minas Gerais, Brazil}
\email{lhcalixto@ufmg.br}

\author{Crystal Hoyt}
\address{C.~Hoyt: Department of Mathematics, Bar-Ilan University, Ramat Gan 52900, Israel}
\email{math.crystal@gmail.com}

%\thanks{Insert later}

%\thanks{Insert later}


%\date{\today}


\begin{abstract}
The Lie superalgebra $W(\infty)$ is defined to be the direct limit of the simple finite-dimensional Cartan type Lie superalgebras $W(n)$ as $n$ goes to infinity, where $W(n)$  denotes the Lie superalgebra of superderivations of the Grassmann algebra $\Lambda(n)$. The zeroth component of $W(\infty)$ in its natural $\mathbb{Z}$-grading is isomorphic to $\fgl(\infty)$.

In this paper, we initiate the study of the representation theory of $W(\infty)$. We study $\Z$-graded $W(\infty)$-modules, and we introduce a category $\bbT_W$ that is closely related to the Koszul category $\bbTsl$ of tensor $\fsl(\infty)$-modules introduced and studied by Dan-Cohen, Serganova and Penkov. We classify the simple objects of $\bbT_W$ (up to isomorphism). 
We prove that each simple module in $\bbT_W$ is isomorphic to the unique simple quotient of a module induced from a  simple module  in $\bbTgl$, and vice versa, which is analogous to the case for $W(n)$ studied by Serganova. As a corollary,  we find that all simple modules in $\bbT_W$ are highest weight modules with respect to a certain Borel subalgebra. We realize each simple module from $\bbT_W$ as a module of tensor fields, generalizing work of Bernstein and Leites for $W(n)$. We prove that the category $\bbT_W$ has enough injective objects, and for each simple module, we provide an explicit injective module in $\bbT_W$ that contains it.
\end{abstract}

\maketitle \thispagestyle{empty}

% \tableofcontents

%%%%%%%%%%%%%%%%
%
\section*{Introduction}
%
%%%%%%%%%%%%%%%%

The study of Lie algebras that can be obtained as a direct limit of classical finite-dimensional Lie algebras has been an active field of research over the last 20 years \cite{PH22}. These Lie algebras are called locally finite Lie algebras.  The program of understanding the structure and representation theory of such Lie algebras  has been led by Ivan Penkov. The classical locally simple Lie algebras $\fsl(\infty), \fo(\infty), \fsp(\infty)$  are precisely the Lie algebras that appear in Baranov's classification of finitary locally simple Lie algebras over $\mathbb C$ \cite{Ba99}.  These Lie algebras admit several natural categories of modules which turn out to have very rich and useful structure theory \cite{PS11, DPS16,  HPS19, PS19, CoP19}. For example, the category $\bbTsl$ of tensor modules over $\fsl(\infty)$ was introduced and studied  in \cite{DPS16}, where it was shown to be a self-dual Koszul category \cite{DPS16}. The category $\bbTsl$ is used  in \cite{FPS16} to categorify the Boson-Fermion Correspondence. In \cite{HPS19}, a category of integrable $\mathfrak{sl}(\infty)$-modules is used to describe the induced $\mathfrak{sl}(\infty)$-module structure on the complexified reduced Grothendieck groups arising from Brundan's categorical $\mathfrak{sl}(\infty)$-action on the category $\mathcal F_{m|n}^{\mathbb Z}$ of finite-dimensional integral $\mathfrak{gl}(m|n)$-modules and on the BGG category $\mathcal O_{m|n}^{\mathbb Z}$ (see \cite{Br03}).

It is natural to consider the super analogs of locally finite Lie algebras and their categories of modules. The categories of tensor modules over $\fgl(\infty|\infty)$ and $\fosp(\infty|\infty)$ are described in \cite{S14}, and the category of integrable bounded weight modules over classical locally finite Lie superalgebras is studied in \cite{CP22}. Although the study of the representation theory of locally finite Lie superalgebras is still in its infancy, it has already proven to be very useful. In \cite{S14}, Serganova found a natural way to explain the equivalence between the categories of tensor modules over $\fo(\infty)$ and $\fsp(\infty)$ by means of the category of tensor modules over $\fosp(\infty|\infty)$.

Moving to non-classical Lie superalgebras, it is safe to say that the Lie superalgebra $W(n)$ of superderivations of the Grassmann algebra $\Lambda(n)$ is the most fundamental one. This Lie superalgebra appears in Kac's classification list of simple finite-dimensional Lie superalgebras \cite{Kac77}. Bernstein and Leites described irreducible finite-dimensional $W(n)$-modules and realized these representations as modules of tensor fields in \cite{BL81,BL83}.  Serganova studied the category of $\Z$-graded $W(n)$-modules and described the structure of induced modules in \cite{S05}.

In this paper, we define the Lie superalgebra $W(\infty)$ to be the direct limit of the Lie superalgebras $W(n)$  as $n$  goes to infinity, with respect to the natural embedding of $W(n)$ into $W(n+1)$. So by definition, $W(\infty)$ is a locally finite Lie superalgebra that is locally simple, and hence simple. The Lie superalgebra $W(\infty)$  has a natural $\mathbb Z$-grading such that the zeroth graded component is isomorphic to $\fgl(\infty)$. The Lie algebra $\fgl(\infty)$ has a corank $1$ simple ideal $\fsl(\infty)$. 

In the current paper we initiate the study of the representation theory of the Lie superalgebra $W(\infty)$. Our main goal is to introduce and investigate categories $\bbT_W$, $\bbT_W^{\geq}$ and $\bbT_W^{\leq}$ of  $W(\infty)$-modules that are analogs of the category $\bbT_{\fsl(\infty)}$. The objects in these categories are $\Z$-graded $W(\infty)$-modules that satisfy some nice properties that resemble those that finite-dimensional modules over $W(n)$ have. We prove that the simple objects of $\bbT_W$ and $\bbT_W^{\leq}$ coincide and that there is a one-to-one correspondence between such simple modules and simple tensor modules of $\fgl(\infty)$ (see Theorem~\ref{theo: simples Lminus} and Corollary~\ref{cor: coincide}). To obtain this bijection we rely on the fact that every simple object in $\bbT_W$ is parabolically induced from a simple object in $\bbT_{\fgl(\infty)}$ (Proposition~\ref{prop:inv.non-zero}). This is analogous to the finite-dimensional situation where every simple finite-dimensional $W(n)$-module can be obtained as a quotient of a parabolically induced simple finite-dimensional $\fgl(n)$-module \cite{S05}. As a corollary, we obtain that simple modules in $\bbT_W$ are highest weight modules with respect to an appropriately chosen Borel subalgebra and that such modules can be parameterized by pairs of partitions. 

In Section~\ref{sex:tensor_modules} we define modules of tensor fields for $W(\infty)$ and we prove that every simple module in $\bbT_W$ is isomorphic to a submodule of a module of tensor fields (Theorem~\ref{thm:simples_are_tensor_modules}). This result gives a realization of all simple modules of $\bbT_W$ and it is analogous to \cite[Theorem~3]{BL81}. In Section~\ref{Section 7} we construct a functor $\Gamma$ from the category $\g\Mod$ of all $\g$-modules to the category $\bbT_W$, which allows us to prove that the category $\bbT_W$ has enough injectives (Proposition~\ref{prop:enough_inj}). We show that combining $\Gamma$ with a certain coinduction functor provides a way to construct injective objects of $\bbT_W$ from injective objects in the category of integrable $\fgl(\infty)$-modules (Corollary~\ref{cor:constructing_inj}). Finally, we use the functor $\Gamma$ along with  Theorem~\ref{thm:simples_are_tensor_modules} to explicitly construct, for every simple module in $\bbT_W$, an injective module of $\bbT_W$   which contains it (Theorem~\ref{prop:injective_for_simple}).

This new category $\bbT_W$ that we introduce in this paper opens up a number of directions for further research, and our results indicate that $\bbT_W$ may be as rich as $\bbT_{\fsl(\infty)}$. We conclude this section by listing some problems to be addressed in future works.
\begin{enumerate}
\item Is $\bbT_W$ Koszul? What is the block decomposition of $\bbT_W$? In order to answer these questions we first need to understand the extension groups between simple modules of this category, which is known to be a very hard problem even in the finite-dimensional setting.
\item Is $\bbT_W$ a highest weight category in the sense of \cite{CPS88}?
\item The bijection we prove between simple objects of $\bbT_W$ and of $\bbT_{\fgl(\infty)}$ raises the question as to whether there is an equivalence between these two categories, as was the case in \cite{S14}.
\end{enumerate}
\medskip

%\paragraph{\textbf{Notation}} Insert later

\iftoggle{detailsnote}{
\medskip

%\paragraph{\textbf{Note on the arXiv version}} For the interested reader, this arXiv version of this paper includes hidden details of some straightforward computations and arguments that are omitted in the pdf file.  These details can be displayed by switching the \texttt{details} toggle to true in the tex file and recompiling.
}{}

\medskip

\paragraph{\textbf{Acknowledgements}} We would like to thank Ivan Penkov and Vera Serganova for helpful discussions. The first author was supported by Fapemig Grant (APQ-02768-21) and by CNPq Grant (402449/2021-5). The second author was partially supported by ISF Grant 1221/17 and by ISF Grant 1957/21.

%%%%%%%%%%%%%%%%
%
\section{The Lie algebras $\fgl(\infty)$ and  $\fsl(\infty)$}\label{sec:gl}
%
%%%%%%%%%%%%%%%%

\subsection*{Notation:} The ground field is always assumed to be $\C$. We set $I:=\{1,2,3,\ldots\}$. If $C$ is a category, then we write $M\in C$ whenever $M$ is an object of $C$. 
If $V$ is a $\Z_2$-graded vector space, then $\bar{v}$ denotes the parity of a homogeneous element $v\in V$. By convention, if we write $\bar{v}$ then we assume that $v$ is homogeneous. For a Lie superalgebra $\frak{a}$, let $\bU(\frak{a})$ denote its universal enveloping superalgebra.

\subsection{Construction and root system of  $\fgl(\infty)$ and  $\fsl(\infty)$}
We begin with defining the Lie algebras $\fgl(\infty)$ and  $\fsl(\infty)$ (see e.g. \cite{PH22}).
Let $V, V_*$ be countable-dimensional vector spaces over $\C$. Fix a non-degenerate pairing $\langle\cdot , \cdot \rangle:V\otimes V_*\to \C$ and  bases $\{\xi_i\}_{i\in I}$ for $V$, $\{\partial_i\}_{i\in I}$ for $V_*$ such that $\langle \xi_i,\partial_j\rangle=\delta_{i,j}$ for all $i,j\in I$. Then $\fgl(\infty):=V\otimes V_*$ has a Lie algebra structure such that
	\[
[\xi_i\otimes \partial_j,\xi_k\otimes \partial_\ell]=\langle \xi_k,\partial_j\rangle\xi_i\otimes \partial_\ell -\langle \xi_i,\partial_\ell\rangle \xi_k\otimes \partial_j.
	\]
Moreover, $V$ and $V_*$ are the natural and conatural $\fgl(\infty)$-modules, respectively, and they satisfy
	\[
(\xi_i\otimes \partial_j)\cdot \xi = \langle\xi, \partial_j\rangle\xi_i,\quad (\xi_i\otimes \partial_j)\cdot \partial = -\langle\xi_i, \partial\rangle\partial_j,\text{ for all }\xi,\xi_i\in V,\  \partial,\partial_j\in V_*.
	\]
In what follows, we write $\xi_i\partial_j$ for the element $\xi_i\otimes\partial_j$.

We can identify $\fgl(\infty)$ with the
space of infinite matrices $\left(a_{ij}\right)_{i,j\in I}$
with finitely many nonzero entries, where the vector $\xi_i\partial_j$ 
corresponds to the matrix $E_{i,j}$ with $1$ in the $i,j$-position and zeros elsewhere. 
Under this identification, $\langle\cdot,\cdot\rangle$ is the trace map on $\mathfrak{gl}(\infty)$.

The Lie algebra $\fsl(\infty)$ is defined to be the kernel of $\langle\cdot,\cdot\rangle$. Once can also realize $\fsl(\infty)$ as a direct limit  of finite-dimensional Lie algebras $\underrightarrow{\lim}\  \fsl(n)$. So the Lie algebra $ \fsl(\infty)$  is locally simple, and hence simple.
Note however that the exact sequence
$$
0\to \fsl(\infty)\to \fgl(\infty)\to\mathbb{C}\to 0
$$
does not split, since the center of  $\fgl(\infty)$ is trivial.

We realize $\fgl(\infty)$ as the direct limit $\underrightarrow{\lim}\  \fgl(n)$ as follows. As a vector space, $\fgl(n)=V_n\otimes V^*_n$, where  $\{\xi_1,\ldots,\xi_n\}$ and $\{\partial_1,\ldots,\partial_n\}$ are fixed dual bases for  $V_n$ and $V^*_n$, respectively. Then we have obvious embeddings $V_n\hookrightarrow V_{n+1}$, $V^*_n\hookrightarrow V^*_{n+1}$, $\fgl(n)\hookrightarrow\fgl(n+1)$ given on  basis vectors by $\xi_i\to\xi_i$ and $\partial_j\to \partial_j$.  Then $\fgl(\infty)=\cup_{n\in \Z_{\geq 2}} \fgl(n)$, and we identify $\fgl(n)$ with its image under the defining map.

We fix the Cartan subalgebras $\h_{\fgl(n)}\subseteq \fgl(n)$, $\h_\fgl\subseteq \fgl(\infty)$ and $\h_\fsl\subseteq \fsl(\infty)$, which consist of diagonal matrices. Then $\h_{\fgl(n)}=\h_\fgl\cap\fgl(n)$ and $\h_\fsl=\h_\fgl\cap\fsl(\infty)$. If we choose the standard basis $\{\varepsilon_1,\varepsilon_2,\varepsilon_3,\ldots\}$ of $\h_*$, then the root system of $\fa=\fgl(\infty)$, $\fsl(\infty)$ is $\Delta=\{\varepsilon_i-\varepsilon_j \mid i\neq j \in I\}$, where the elements $\xi_i\partial_j = E_{i,j}$ are root vectors for $\varepsilon_i-\varepsilon_j$.

We let $\fb(\prec)^0$ be  the Borel subalgebra  of $\fa=\fgl(\infty)$, $\fsl(\infty)$ corresponding to the following linear order $\prec$ on our index set $I$:
\begin{equation}\label{prec def}
1\prec 3\prec 5 \prec \cdots \prec 6\prec 4\prec 2,
\end{equation}
that is, $\fb(\prec)^0=\h\oplus\n^0$ where $\n^0=\oplus_{\alpha\in\Delta^+}\fa_{\alpha}$ and
\begin{equation}\label{pos roots b0}
\Delta^+:=\{\varepsilon_i-\varepsilon_j\mid i<j \text{ odd}\}\cup\{\varepsilon_i-\varepsilon_j\mid j<i \text{ even}\}\cup\{\varepsilon_i-\varepsilon_j\mid i \text{ odd, } j \text{ even}\}.
\end{equation}
We let $\fb(\prec)_n^0$ be the Borel subalgebra  of $\fgl(n)$ defined by  $\fb(\prec)^0\cap \fgl(n)$, using our fixed embedding  $\fgl(n)\hookrightarrow\fgl(\infty)$.  
With respect to the  Borel subalgebra  $\fb(\prec)^0$, both the natural module $V$ and the
conatural module $V_*$ are highest weight modules, and they have highest weights $\varepsilon_1$ and $-\varepsilon_2$, respectively.

%%%

\subsection{Tensor modules over $\fgl(\infty)$ and $\fsl(\infty)$}\label{mod gs} Let $\fa=\fgl(\infty)$, $\fsl(\infty)$. An $\fa$-module is called a {\em tensor module} if it is isomorphic to a subquotient of a finite direct sum of modules of the form $V^{\otimes p_i}\otimes V_{*}^{\otimes q_i}$ for some $p_i,q_i \in\Z_{\geq 0}$ \cite{DPS16}. The module $ V^{\otimes p}\otimes  V_{*}^{\otimes q}$  is not semisimple if $p,q>0$.  The simple subquotients of the modules $V^{\otimes p}\otimes V_*^{\otimes  q}$, $p,q\in\Z_{\geq 0}$   can be parameterized by two Young diagrams $\lambda,\mu$,   and we will denote them  $ V_{\lambda,\mu}$ (see Theorem~\ref{th:PSt}). 

By Schur--Weyl duality,  we have the following isomorphism of $\fgl(\infty)\times (S_p \times S_q)$-modules:

$$V^{\otimes p}\otimes V_*^{\otimes  q}\simeq\bigoplus_{|\lambda|=p}\ \bigoplus_{|\mu|=q}
(\mathbb{S}_{\lambda}( V)\otimes \mathbb{S}_{\mu}( V_*))\otimes(Y_\lambda\boxtimes Y_\mu),$$
where  $\mathbb{S}_{\lambda}$ denotes the Schur functor corresponding to the Young diagram  $\lambda$, the size of $\lambda$ is  $|\lambda|:=\sum \lambda_i$, and $Y_\lambda\boxtimes Y_\mu$ is the outer tensor product of irreducible $S_p$- and $S_q$-modules.

Then $\tilde{V}_{\lambda,\mu}:=(\mathbb{S}_{\lambda}( V)\otimes \mathbb{S}_{\mu}( V_*))$ are indecomposable $\fa$-modules, and their socle filtration was described by Penkov and Styrkas in \cite{PSt11}. We recall that the {\em socle} of a module $M$, denoted $\soc M$,
is the largest semisimple submodule of $M$, and that the {\em socle filtration} of $M$ is defined inductively by $\soc^0 M:=\soc M$ and $\soc^i M:=p_i^{-1}(\soc (M/(\soc^{i-1}M)))$, where $p_i:M\to
M/(\soc^{i-1} M)$ is the natural projection. The layers of the socle filtration are denoted by $\overline{\soc}^i M :=\soc^i M /\soc^{i-1} M $.

The following theorem is from \cite[Theorem 2.3]{PSt11}.

\begin{theo}[Penkov, Styrkas]\label{th:PSt}  Let $\fa=\mathfrak{gl}(\infty)$,
$\mathfrak{sl}(\infty)$. 
The socle filtration of $\tilde{V}_{\lambda,\mu}$ has the following layers:
\begin{equation}\label{ch7:socle filtration}
\overline{\soc}^{k}\tilde{V}_{\lambda,\mu}=\bigoplus_{|\gamma|=k}
N^{\lambda}_{\gamma,\lambda'}N^{\mu}_{\gamma,\mu'}V_{\lambda',\mu'},
\end{equation}
where $N^{\lambda}_{\gamma,\lambda'}$ $N^{\lambda}_{\gamma,\lambda'}$  denotes the Littlewood--Richardson coefficients determined by the relation $s_\lambda \, s_\mu =
\sum_\nu N_{\lambda,\mu}^\nu \, s_\nu$ for the Schur symmetric polynomials $s_\lambda$, $s_\mu$, $s_\nu$.
\end{theo}

In particular, the indecomposable $\fa$-module  $\tilde{V}_{\lambda,\mu}$ has an irreducible socle, denoted $V_{\lambda,\mu}$.
For any partitions  $\lambda,\mu$, the $\fa$-module  $V_{\lambda,\mu}$  is an irreducible highest weight module with respect to the Borel subalgebra $\fb(\prec)^0$ and has
highest weight 
\begin{equation}\label{eq hw}
\chi:=\sum_{i\in\Zp}\lambda_{i} \varepsilon_{2i-1} -\sum_{j\in\Zp}\mu_{j} \varepsilon_{2j} 
\end{equation}
(see \cite[Theorem 2.1]{PSt11}).

One corollary of Theorem~\ref{th:PSt} is that  each simple subquotient  of  $ V^{\otimes p}\otimes  V_{*}^{\otimes q}$ is isomorphic to a (simple) submodule of $ V^{\otimes p'}\otimes  V_{*}^{\otimes q'}$ for some $p',q'$.

%%%

\subsection{The categories $\bbTgl$ and $\bbTsl$}

A category of tensor modules over $\fsl(\infty)$ was introduced and studied in \cite{DPS16}. Here we recall this construction, and define a similar category for $\fgl(\infty)$.  First, we need the following definitions.

Let $\fa=\fgl(\infty)$, $\fsl(\infty)$. An $\fa$-module $M$ is called \emph{integrable} if for every $x\in \fa$ and $m\in M$, we have
	\[
\dim \Span_\C \{x^im\mid i\in  \Zp\}<\infty.
	\]
For each $n\in \Zp$,  let $\ft_n'$ be the subalgebra of $\fgl(\infty)$  generated by  root vectors as follows: 
\begin{equation}\label{eq tn}
\ft_n':=\left\langle\xi_{i}\partial_j =E_{i,j} \mid j,i \geq n \right\rangle.
\end{equation}
\begin{defin} 
We say that a subalgebra $\fk\subset \fa$ has \emph{finite corank} in $\fa$ if $(\ft_n'\cap\fa)\subset\fk$ for some $n\gg 0$.  
\end{defin}

\begin{rem} For $\fa=\fsl(\infty)$, a subalgebra $\fk\subset \fa$ has finite corank if and only if $\fk$ contains the commutator subalgebra of the centralizer of some finite-dimensional subalgebra of $\fa$.
\end{rem}

We say that an $\fa$--module $M$ satisfies the \emph{large annihilator condition (l.a.c)} if, for every $m\in M$, the subalgebra $\Ann_{\fa}(m)$ has finite corank.

\begin{defin}\label{def Tgl}
The category $\bbTgl$ is the full subcategory of $\fgl(\infty)$-mod consisting of modules $M$ that satisfy the following conditions:
\begin{enumerate}[(1)]
\item $M$ has finite length;
\item $M$ is integrable;
\item $M$ satisfies the l.a.c. 
\end{enumerate}
The category $\bbTsl$ is defined similarly, by replacing $\fgl(\infty)$ with $\fsl(\infty)$.
\end{defin}

Then $\bbTgl$ and  $\bbTsl$ are  abelian symmetric monoidal categories.
It was proved in \cite[Corollary 4.6]{DPS16} that the category  $\bbTsl$  consists of tensor modules.
Moreover, the modules $ V^{\otimes p}\otimes  V_{*}^{\otimes q}$, $p,q\in\mathbb{Z}_{\geq 0}$ are injective in the  category $\bbTsl$, and every indecomposable injective object of $\bbTsl$  is isomorphic to an indecomposable direct summand of  $ V^{\otimes p}\otimes  V_{*}^{\otimes q}$  for some $p,q\in\mathbb{Z}_{\geq 0}$  \cite{DPS16}. So by Theorem~\ref{th:PSt}, an indecomposable injective module in $\bbTsl$  is isomorphic to
 $\tilde{V}_{\lambda,\mu}:=(\mathbb{S}_{\lambda}( V)\otimes \mathbb{S}_{\mu}( V_*))$ for some $\lambda,\mu$.
Moreover, by \cite[Theorem 3.4]{DPS16}, every module $M$ in $\bbTsl$ is an $\h_\fsl$-weight module (since $\h_\fsl$ is a ``splitting'' Cartan subalgebra).

\begin{lem}\label{lem:restTglInTsl}
The restriction functor
    \[
\operatorname{R}_{\fsl}:\fgl(\infty)\Mod\to \fsl(\infty)\Mod,\quad \operatorname{R}_{\fsl}(M):=M|_{\fsl}
    \]
restricts to a well defined functor from $\bbTgl$ to $\bbTsl$.
\end{lem}
\begin{proof} Let $M\in\bbTgl$, and consider $\operatorname{R}_{\fsl}(M):=M'$ its restriction to $\fsl(\infty)$.
Trivially, $M'$ satisfies conditions~(2) and (3) of Definition~\ref{def Tgl}. To show that (1) holds, it suffices to prove that $M$ being simple implies that $M'$ is simple. For this, let $m\in M$ be a non-zero vector and apply the l.a.c. to choose $n\gg 0$ so that $E_{n,n}m=0$. Then, by the PBW Theorem, we have
	\[
M = \bU(\fgl(\infty))m = \bU(\fsl(\infty))m,
	\]
since $\fgl(\infty) =\fsl(\infty)\oplus\C E_{n,n}$. Thus $M'$ is simple, and the statement follows.
\end{proof}

\begin{lem}\label{lem:sl-weight=gl-weight}
Let $M$ be a $\fgl(\infty)$-module satisfying the l.a.c. and such that $\operatorname{R}_{\fsl}(M)$ is an $\h_{\fsl}$-weight module. Then $M$ is an $\h_\fgl$-weight module. Moreover, $\mu,\lambda\in \Supp M$ are equal if and only if $\mu|_{\h_{\fsl}} = \lambda|_{\h_{\fsl}}$. In particular, $\Supp \operatorname{R}_{\fsl}(M) = \{\lambda':=\lambda|_{\h_{\fsl}}\mid \lambda\in \Supp M\}$ and $\operatorname{R}_{\fsl}(M)_{\lambda'}=M_{\lambda}$.
\end{lem}
\begin{proof}
For any $\h_\fsl$-weight vector $m$, we can use the l.a.c. to find $E_{n,n}\in \h_{\fgl}$ for $n\gg 0$ such that $E_{n,n} m=0$. Since $\h_\fgl = \h_\fsl\oplus \C E_{n,n}$, we get that $m$ is also an $\h_\fgl$-weight vector.

Now, let $\mu,\lambda\in \Supp M$ such that $\mu |_{\h_\fsl} = \lambda |_{\h_\fsl}$. Again by the l.a.c., for nonzero $m\in M_{\mu}$, we can choose $n\gg 0$ so that $0=E_{n,n} m= \mu(E_{n,n}) m$, and similarly for $\lambda$.  Thus, we can choose  $n\gg 0$, such that $\mu(E_{n,n}) = \lambda(E_{n,n})=0$. Since $\h_\fgl = \h_\fsl\oplus \C E_{n,n}$ this implies $\mu = \lambda$. Finally, since $M$ and $\operatorname{R}_{\fsl}(M)$ have the same underlying vector space, this completes the proof.
\end{proof}

\begin{lem}\label{lem:tensorModules=weightModules}
If $M\in \bbTgl$, then $M$ is an $\h_\fgl$-weight module.
\end{lem}
\begin{proof}
Since  $\operatorname{R}_{\fsl}(M)\in\bbTsl$ by Lemma~\ref{lem:restTglInTsl}, we have
that $\operatorname{R}_{\fsl}(M)$ is an $\h_\fsl$-weight module. Now the statement follows from Lemma~\ref{lem:sl-weight=gl-weight}
\end{proof}

\begin{prop}\label{prop T eq}
The categories $\bbTgl$ and $\bbTsl$ are equivalent. 
\end{prop} 
\begin{proof} 
Let $M'\in \bbTsl$.  Then, by \cite[Corollary~4.6]{DPS16}, $M'$ is isomorphic to a submodule of a finite direct sum of copies of the tensor algebra $T(V\oplus V_*)$, and so $M'$ inherits a natural $\fgl(\infty)$-module structure. It is clear that $M'$ with this $\fgl(\infty)$-module structure is an object of $\bbT_{\fgl(\infty)}$. Moreover, since $\fgl (\infty) = \fsl (\infty)\oplus \C E_{n,n}$ for every $n\in \N$, this is the only way to extend the $\fsl(\infty)$-module structure of $M'$ to a $\fgl(\infty)$-module structure in such way that the resulting $\fgl(\infty)$-module satisfies the l.a.c.. Indeed,
Lemma~\ref{lem:sl-weight=gl-weight} shows that the action of $\h_\fsl$ on $M'$ determines the action of $\h_\fgl$. We let $M\in \bbTgl$ denote $M'$ viewed as a $\fgl(\infty)$-module endowed with such structure, and define a functor $F:\bbTsl \to \bbTgl$ by setting $F(M') = M$, and $F(f)=f$ for all $f\in \Hom_{\bbTsl}(M, N)$. 
Then the functors $F$ and $\operatorname{R}_{\fsl}$ are inverse to each other.
\end{proof}

We denote by $\tbbTgl$ the Grothendieck envelope of the category $\bbTgl$,  that is, $\tbbTgl$ is the full subcategory of
$\fgl(\infty)$-mod with objects being arbitrary sums of objects in $\bbTgl$ (see \cite{ChP17}). We let $\tbbTsl$ denote the Grothendieck envelope of $\bbTsl$. It follows that  $\tbbTgl$ and $\tbbTsl$ are equivalent. Note that  $\bbTgl$ and  $\tbbTgl$ have the same simple objects.

We have the following characterization of  $\tbbTgl$.

\begin{prop}\label{lem Groth T} The category
 $\tbbTgl$ is the full subcategory of $\fgl(\infty)$-mod  consisting of modules $M$ that satisfy the following  conditions:
\begin{enumerate}[(1)]
\item $M$ is integrable;
\item $M$ satisfies the l.a.c. 
\end{enumerate}
\end{prop}
\begin{proof} First, if $M\in\tbbTgl$ then $M$ satisfies (1) and (2), since $M=\sum M_\gamma$ where each $M_\gamma\in\bbTgl$ satisfies these conditions.

Next, suppose $M$ is a $\fgl(\infty)$-module satisfying conditions (1) and (2).  We claim  $M=M|_{\fsl(\infty)}$ lies in the category $\mathcal{OLA}_\fb$, as defined in \cite[Section~3]{PS19}. To show this we need to verify the following conditions: (i) $M$ satisfies the l.a.c.;
(ii) $M$ is $\h_{\fsl}$-semisimple; (iii) every $x\in\n^0$ acts locally nilpotently on $M$. We get (i) from (2). To prove that (1) and (2) imply (ii), we write $M=\sum M_\gamma$, where all $M_\gamma$'s are finitely generated $\fsl(\infty)$-modules. Now, we notice that the proof of $(3) \Rightarrow (2) \Rightarrow (1)$ in \cite{DPS16}[Theorem~3.4] does not require an $\fsl(\infty)$-module to be of finite length, but only finitely generated. Therefore, we can apply this result to each $M_\gamma$ to conclude that each such module satisfies (ii). It follows that $M$ itself satisfies (ii) (see Remark~\ref{rem:MIntbbTglIsWeightModule} below). Finally, since $M$ is $\h_{\fsl}$-semisimple, condition (iii) follows from integrability. Indeed, if $x\in(\n(\prec)^0)_{\alpha}$ and  $\in M_{\mu}$, then integrability implies that $x^n m=0$ for $n\gg 0$, since $x^k m\in M_{\mu+k\alpha}$. The claim follows.

Since each $M_\gamma$ is a finitely generated $\fsl(\infty)$-module in $\mathcal{OLA}_\fb$,
it follows from \cite[Proposition~4.17]{PS19} that each such module is a finite-length $\fsl(\infty)$-module (and hence finite-length  $\fgl(\infty)$-module by the l.a.c.). Therefore,  each  $M_\gamma$ lies in $\bbTgl$, and we  conclude that $M\in\tbbTgl$, since $\tbbTgl$ is closed under arbitrary sums.
\end{proof}

The claim and proof of Proposition~\ref{lem Groth T} hold if we replace $\fgl(\infty)$ by $\fsl(\infty)$ everywhere.

\begin{rem}\label{rem:MIntbbTglIsWeightModule}
If $M\in \tbbTgl$, then we claim that $M$ is a $\h_\fgl$-weight module. Indeed, notice that arbitrary direct sums, submodules and quotients of weight modules are weight modules. Then the claim follows from the fact that arbitrary sums are quotients of arbitrary direct sums.
\end{rem}

%%%%%%%%%%%%%%%%
%
\section{The Lie superalgebra $W(\infty)$}\label{sec:W}
%
%%%%%%%%%%%%%%%%

In this section, we define $W(\infty)$ to be the direct limit $W(\infty) := \lim\limits_{\longrightarrow} W(n)$ of the finite-dimensional Cartan type Lie superalgebras $W(n)$, with respect to the obvious embedding. The Lie superalgebra $W(\infty)$ is a proper subalgebra of  $\Der(\Lambda(\infty))$, the Lie superalgebra of all superderivations of the (infinite) Grassmann algebra  $\Lambda(\infty)$.
We also describe the root system of $W(\infty)$ with respect to our choice of Cartan subalgebra.

\subsection{The Grassmann algebra $\Lambda(\infty)$} 

The Grassmann algebra $\Lambda(n)$ is by definition the free commutative (unital) superalgebra with $n$ odd generators $\xi_1,\ldots,\xi_n$.  In particular, $\xi_i^2=0$ for  $i\in \{1,2,\ldots,n\}$. 
We define the (infinite) Grassmann algebra  $\Lambda(\infty)$ to be the free commutative superalgebra with countably many odd generators $\{\xi_i\}_{i\in I}$ satisfying $\xi_i^2=0$. Clearly, $\Lambda(\infty) = \lim\limits_{\longrightarrow} \Lambda(n)$, where $\Lambda(n)\hookrightarrow \Lambda(n+1)$ in the obvious way.

The Grassmann algebra  $\Lambda(\infty)$  has a natural $\Z$-grading $\Lambda(\infty)=\bigoplus_{k\geq 0} \Lambda(\infty)^k$,  obtained by setting $\deg \xi_i=1$ for each generator $\xi_i$, $i\in I$,   which is compatible with the superalgebra grading. 
We let $$\cE:=\{\underline{e}: I\to \{0,1\}\mid \Supp \underline{e}<\infty\}$$ and, for any $\underline{e}\in \cE$, we set $$\underline{\xi}^{\underline{e}}=\xi_{i_1}\cdots \xi_{i_n}:=\xi_{i_1}\wedge\cdots\wedge \xi_{i_n},$$ where $\Supp \underline{e}=\{i_1,\ldots, i_n\}$ and $i_1<\cdots < i_n$. Then $\deg \underline{\xi}^{\underline{e}}=|\Supp\underline{e}|=\sum_{i\in I} \underline{e}(i)$ and
the set $\{\underline{\xi}^{\underline{e}}\mid \underline{e} \in \cE\}$ is a $\C$-basis for $\Lambda(\infty)$. 
So 
 $\Lambda(\infty)^k=\Lambda^{k+1}(V)=\langle \xi_{i_1}\cdots \xi_{i_k}\mid i_1<\cdots < i_k\rangle_\C$ and $\Lambda(\infty)^0=\C$.

\subsection{The Lie superalgebra $W(\infty)$ }
Now let $W(n)$ denote the Lie superalgebra $\Der(\Lambda(n))$ of superderivations of $\Lambda(n)$. 
An element of $W(n)$ can be written as $\sum_{i=1}^n  P_i\partial_i$, where $P_i\in\Lambda(n)$ and $\partial_i$, for $i=1,\ldots,n$, is the derivation defined by 
$$\partial_i(\xi_j)=\delta_{ij}.$$
 Setting  $\deg \partial_i=-1$ yields a $\Z$-grading 
$$W(n)=\bigoplus_{k=-1}^{n-1} W(n)^k,$$
 with the property that $W(n)^0$ is isomorphic to $\fgl(n)$ and $W(n)^{-1}$ is isomorphic to the conatural module of $\fgl(n)$.

We define the Lie superalgebra $W(\infty) $ to be the direct limit $W(\infty) :=\lim\limits_{\longrightarrow} W(n)$, where the embedding $W(n)\hookrightarrow W(n+1)$ is defined in the obvious way, by sending the generators  $\xi_1,\ldots,\xi_n,\partial_1,\ldots,\partial_n$ of $W(n)$ to the generators of  $W(n+1)$ that have the same name. Then  $W(\infty)$ is an infinite-dimensional locally finite Lie superalgebra, that is locally simple, and hence simple. 

Let $\Der(\Lambda(\infty))\subseteq \End_\C(\Lambda(\infty))$ be the Lie superalgebra of all superderivations of $\Lambda(\infty)$.  The Lie superalgebra $\Der(\Lambda(\infty))$ is quite large, and  $W(\infty)$ can be identified with the subspace of $\Der(\Lambda(\infty))$ spanned by all elements of the form
$
P(\xi)\partial_i
$
with $P(\xi)\in\Lambda(\infty)$, $i\in I$.

 The underlying super vector space of $W(\infty)$ is $\Lambda(V)\otimes V_*$, and the set 
$$\{\underline{\xi}^{\underline{e}} \partial_j\mid \underline{e} \in \cE,\  i\in I\}$$
 is a  $\C$-basis for $W(\infty)$.
  The element $\underline{\xi}^{\underline{e}} \partial_i$ has degree $(\sum \underline{e}(k))-1$.
 The bracket of homogeneous elements is given by
 $$
 [\underline{\xi}^{\underline{a}} \partial_j,
 \underline{\xi}^{\underline{b}}\partial_\ell]
 = \underline{\xi}^{\underline{a}} \partial_j \circ
 \underline{\xi}^{\underline{b}}\partial_\ell
 - (-1)^{\deg (\underline{\xi}^{\underline{a}} \partial_j) \deg(
 \underline{\xi}^{\underline{b}}\partial_\ell)}
 \underline{\xi}^{\underline{b}}\partial_\ell \circ
 \underline{\xi}^{\underline{a}}\partial_j
 $$ 
 and  is of the form
 \begin{equation}\label{bracket form}
 [\underline{\xi}^{\underline{a}} \partial_j,
 \underline{\xi}^{\underline{b}} \partial_\ell]
 =\pm\ \delta_{\underline{b}(j),1}\ \underline{\xi}^{\underline{a}} \underline{\xi}^{\underline{b}-\underline{e}_j}\partial_\ell \ \pm \
 \delta_{\underline{a}(\ell),1}\ \underline{\xi}^{\underline{b}} \underline{\xi}^{\underline{a}-\underline{e}_\ell}\partial_j,
\end{equation} 
 where $\underline{a},\underline{b}\in\cE$ and $j,\ell \in \Z_{>0}$ (see e.g. \cite{Gav15}).
The  $\Z$-grading on $W(\infty)$ is
	\[
W(\infty)=\bigoplus_{k\geq -1} W(\infty)^k,\text{ where } W(\infty)^k=\operatorname{span}\{ P_j\partial_j\mid P_j\in \Lambda(\infty)^{k+1}, j\in I\},
	\]
and is compatible with the superalgebra grading, i.e., $W(\infty)_{\bar 0}=\bigoplus_{k\geq 0} W(\infty)^{2k}$.

Now the zeroth component $W(\infty)^0$ is isomorphic to $\fgl(\infty)$, and we will identify  $W(\infty)^{0}$ with $\fgl(\infty)$   by sending $\xi_i\partial_j$ to $E_{i,j}$.  Then  $\fgl(\infty)$ acts on each graded component $W(\infty)^k$ by restricting the adjoint action to $W(\infty)^{0}$. For each $k\geq -1$, we have a $\fgl(\infty)$-module isomorphism $$W(\infty)^{k}\cong\Lambda^{k+1}(V)\otimes V_*,$$ while for $k\leq -2$, we have
$W(\infty)^{k}=\{0\}$.

We fix a Cartan subalgebra $\h$ of $W(\infty)$  to be the subspace spanned by $\xi_i\partial_i$, for $i\in I$. Then $\h$ is also a Cartan subalgebra of  $W(\infty)^{0}$ and, under our identification with $\fgl(\infty)$, we have that $\h=\h_\fgl$ is the space of diagonal matrices. The root system of $W(\infty)$ is 
\[
\Delta= \{\varepsilon_{i_1}+\cdots + \varepsilon_{i_k}- \varepsilon_{j} \mid i_1 <\cdots < i_k,\ j\neq i,\ k\geq 0\}\cup
 \{\varepsilon_{i_1}+\cdots + \varepsilon_{i_k}\mid i_1 <\cdots < i_k,\ k\geq 1\},
\]
where $\xi_{i_1}\cdots \xi_{i_k}\partial_j\in W(\infty)_{\varepsilon_{i_1}+\cdots + \varepsilon_{i_k}- \varepsilon_{j}}$.

%%%%%%%%%%%%%%%%
%
\section{Categories of $W(\infty)$-modules}\label{sec: W mod}
%
%%%%%%%%%%%%%%%%

From now on let $\g=W(\infty)$,  let  $\g^k=W(\infty)^k$, and let $\g_n$ be the image of $W(n)$ in $\g$ under the natural inclusion so that  $\g=\bigcup\g_n$. Then, under our identification, we have $\g^0=\fgl(\infty)$,  $\g_n^0=\fgl(n)$ and $\g^0=\bigcup\g_n^0$.

In this section, we introduce three categories of $\g$-modules: $\bbT_W$, $\bbT_W^{\geq}$ and $\bbT_W^{\leq}$.
First, we need some definitions. 
For each $n\in \Zp$,   let $\ft_n$ be the subalgebra of $\g$ generated by root vectors as follows:
$$
\ft_n:=\left\langle\partial_j,\ \xi_{i_1}\cdots \xi_{i_k}\partial_j\mid j,i_t \geq n\right\rangle.
$$  Then $\ft_n\cap\g^0=\ft_n'$ (see (\ref{eq tn})).

\begin{defin}
Let $\fs$ be a subalgebra of $\g$.  We say that a subalgebra $\fk\subset\g$ has  \emph{finite corank relative to $\fs$} if  for some $n\gg 0$ we have   $(\ft_n\cap\fs) \subset\fk$. 
\end{defin}

\begin{defin} Let $M$ be a $\g$-module, and let $\fs$ be a subalgebra of $\g$.
We say that $M$  satisfies the \emph{large annihilator condition (l.a.c) for $\fs$} if, for every $m\in M$, the subalgebra $\Ann_{\g}(m)$ has finite corank relative to~$\fs$.
\end{defin}

For a $\g$-module $M$, let $M^{\ft_n}=\{m\in M\mid \ft_n\cdot m=0\}$.
\begin{lem}
$M$ satisfies the l.a.c. for $\g$ if and only if $M=\bigcup_{i>0} M^{\ft_n}$.
\end{lem}
\begin{proof}
Now $ m\in M^{\ft_n}  \Leftrightarrow \ft_n\subseteq \Ann_\g(m)$. Hence, for each $m\in M$, the subalgebra $\Ann_\g(m)$ has finite corank in $\g$ if and only $m\in M^{\ft_n}$ for some $n>0$.
\end{proof}

We consider the following subalgebras of $\g$: 
\begin{equation}\label{def g subs}
\g^{\geq}:=\bigoplus_{k\geq 0}\g^k, \hspace{1cm}
\g^>:=\bigoplus_{k> 0}\g^k,
\hspace{1cm}
\g^{\leq}:=\g^{-1}\oplus\g^{0},
\hspace{1cm}
\g^<:=\g^{-1}.
\end{equation}

\begin{defin}\label{def:T_W-Cat}
We define $\bbT_W$ (respectively, $\bbT_W^{\geq}$, $\bbT_W^{\leq}$) to be the full subcategory of $\g$-mod consisting of modules $M$  that satisfy the following  conditions:
\begin{enumerate} [(1)]
\item $M$ has a $\Z$-grading $M=\bigoplus_{k\in\Z}M^k$ that is compatible with the superalgebra grading;
\item $M$ is integrable over $\g^0$;
\item $M$ satisfies the l.a.c. for $\g$ (respectively, for $\g^{\geq}$, $\g^{\leq}$).
\end{enumerate} 
\end{defin}  

Then $\bbT_W$ and  $\bbT_W$ are  abelian symmetric monoidal categories.

\begin{rem}
If $M$ is an object of $\bbT_W$, $\bbT_W^{\geq}$ or $\bbT_W^{\leq}$, then the $\fgl(\infty)$-module $M|_{\g^0}$ obtained by restricting to $\g^0$ is by definition an object of $\tbbTgl$, since $\g^0\subseteq \g^\geq\cap \g^\leq$. It follows from Remark~\ref{rem:MIntbbTglIsWeightModule}
that the $W(\infty)$-module $M$ is an $\h$-weight module.
\end{rem}

\begin{rem}
Clearly, $\bbT_W$ is a subcategory of $\bbT_W^\leq$ and $\bbT_W^\geq$. Moreover, $\bbT_W$ is precisely the full subcategory of $\g$-mod whose objects are in both $\bbT_W^\leq$ and $\bbT_W^\geq$. Indeed, if $M\in\bbT_W^\leq$ and $M\in\bbT_W^\geq$, then for each $m\in M$ we can pick $n\gg 0$ so that both $\ft_n\cap \g^\geq$ and $\ft_n\cap \g^\leq$ are contained in $\Ann_\g(m)$, which implies that $\ft_n\subseteq \Ann_\g(m)$.
\end{rem}

\begin{example}\label{ex:TW}
Here are a few examples of modules in $\bbT_W$: the trivial module $\C$, the \emph{natural module} $\Lambda(V)$, the simple natural module $\Lambda(V)_+:=\Lambda(V)/\C$, and the adjoint module $\g$.
It is clear that condition (1) holds for these modules, while  (2) follows from the next easy lemma. Condition (3) can be verified by direct computation.  \hfill $\bigcirc$ \end{example}

\begin{lem}
Suppose $M=\bigcup M_n$ is a $\g$-module such that each $M_n$ is a finite-dimensional $\g_n$-module. Then   $M$ is integrable over $\g^0$.
\end{lem}
\begin{proof} 
Let $m\in M$ and $x\in\g^0$.  Choose $N\gg 0$ such that $m\in M_n$ and $x\in \g_n^0$. Then $\Span_\C \{x^i m\mid i\in  \Zp\} \subseteq M_n$, which is finite dimensional. 
\end{proof}

\begin{rem} Note that integrability of a $\g$-module over the even part $\g_{\bar 0}$ implies integrability over $\g$, since for any odd element $x\in \g_{\bar 1}$, we have $x^2 = (1/2)[x,x]\in \g_{\bar 0}$. However,  integrability over the zeroth component $\g^0$ does not imply integrability over $\g$.  Consider, for example, the induced module  $K^-(X)=\Ind_{\g^\leq}^{\g} (X)$ where $X\in\bbTgl$ and $\g^<$ acts trivially on $X$. Then $K^-(X)$ is integrable over $\g^0$ but not over $\g$, since $\bU(\g^>)$ acts freely on it.
\end{rem}

\section{Induced modules} 
In this section, we study $W(\infty)$-modules which are induced from $\fgl(\infty)$-modules.

\begin{defin}
Let $X$ be a module over $\g^0=\fgl(\infty)$, and let $\g^>$ (respectively, $\g^<$) act trivially on it. We define the induced $\g$-modules
\begin{align*}
& K^+(X):=\Ind_{\g^\geq}^{\g} (X) = \bU(\g)\otimes_{\bU(\g^\geq)} X, \\
& K^-(X):=\Ind_{\g^\leq}^{\g} (X) = \bU(\g)\otimes_{\bU(\g^\leq)} X.
\end{align*}
\end{defin}

\begin{rem}
It follows from the PBW Theorem that $\bU(\g)$ is a free right $\bU(\g^{\geq })$-module. Thus $K^+(-)$ is an exact functor. The same holds for $K^-(-)$.
\end{rem} 

 It is standard to show that $K^\pm(X)$ admits a unique simple quotient if and only if $X$ is a simple $\g^0$-module. Such a simple quotient will be denoted by $L^\pm(X)$.

We can define $\g_n^\geq$, $\g_n^>$, $\g_n^\leq$ and $\g_n^<$ analogously to (\ref{def g subs}). Let $X_n$ be a $\g_n^0$-module, and extend the action to  $\g_n^>$ (respectively, $\g_n^<$)  trivially. Similarly, we define the induced $\g_n$-modules 
$$
K_n^+(X_n):=\Ind_{\g_n^\geq}^{\g_n} (X_n), \hspace{1cm}
K_n^-(X_n):=\Ind_{\g_n^\leq}^{\g_n} (X_n),
$$
and, for simple $X_n$, their simple quotients are denoted by $L^\pm_n(X_n)$.

We have the following result.

\begin{prop}\label{prop:kac.lim=lim.kac} For each $n\geq 2$, let $f_n : X_n \hookrightarrow X_{n+1}$ be an embedding of $\g_n^0$-modules, and consider the natural embedding $f_n : K^\pm (X_n) \hookrightarrow K^\pm (X_{n+1})$ of $\g_n$-modules, where $f_n(uv) = uf_n(v)$ for every $u\in \bU\left(\g_n\right)$ and $v\in X_n$. 

Then we have an isomorphism of $\g$-modules
$$K^\pm(\varinjlim_n X_n)\cong \varinjlim_n K_n^\pm(X_n),$$
where the limits are taken over the family $\{f_n\}$.
\end{prop}
\begin{proof}
Let $\varphi_{i,j}:X_i \hookrightarrow X_j$ be the family of embeddings defining $X:=\varinjlim X_i$, and let $\varphi_i:X_i\hookrightarrow X$ such that $\varphi_i = \varphi_j\circ \varphi_{i,j}$ for all $i\leq j$. By definition, $\varinjlim K_i^\pm(X_i)$ is defined through the embeddings $\widetilde{\varphi}_{i,j}:=\iota_{i,j}\otimes \varphi_{i,j}$, where $\iota_{i,j}:\bU(\g_i)\hookrightarrow \bU(\g_j)$ are the canonical inclusions. Notice also that $\widetilde{\varphi}_i = \iota_i\otimes \varphi_i$. Define $$\phi_i:\bU(\g_i)\otimes X_i \hookrightarrow \bU(\g)\otimes X$$ by $\phi_i (u\otimes v) = u \otimes [v]$, where $[v]$ is a representative of the class of $v$ in $X$. Notice that $$\phi_j\circ \widetilde{\varphi}_{i,j}(u\otimes v) = \phi_j(u\otimes \varphi_{i,j}(v)) = u\otimes [\varphi_{i,j}(v)] = u\otimes [v] = \phi_j(u\otimes v).$$ Thus, by the universal property of $\varinjlim K_i^\pm(X_i)$, there exists a unique homomorphism of $\g$-modules $$\sigma:\varinjlim K_i^\pm(X_i)\to K^\pm(\varinjlim X_i)$$ such that $\sigma\circ \widetilde{\varphi}_i = \phi_i$ for every $i$. Then, for any $u\otimes [v]\in K^\pm(\varinjlim X_i)$, there is $i\gg0$ for which $u\otimes [v] = \phi_i(u\otimes v) = \sigma\circ\widetilde{\varphi}_i(u\otimes v)$. Hence $\sigma$ is surjective. 

On the other hand, let $w\in \ker \sigma$. Then we can write $w = \widetilde{\varphi}_i(\sum u_i\otimes v_i)$ for linearly independent elements $u_i\in \bU(\g_i)$ for $i\gg 0$. In particular, $$0 = \sigma(w)=\sigma\circ\widetilde{\varphi}_i(\sum u_i\otimes v_i) = \phi_i(\sum u_i\otimes v_i) = \sum u_i\otimes [v_i],$$ and since the $u_i$ are linearly independent, this implies $[v_i]=0$ for every $i$. But $\varphi_i$ is an embedding for every $i$, which implies $v_i=0$, and hence $w=0$. Therefore $\sigma$ is an isomorphism.
\end{proof}

We define a $\Z$-grading on  $X\in \bbTgl$ as follows. Notice that $X$ has  a weight decomposition $X=\bigoplus_{\mu\in \h^*}X_\mu$ and
$$
\Supp X \subset\left\{\mu=\sum\mu_i\varepsilon_i \, \middle\vert \,
\mu_i\in \Z \text{ for all } i\in I, \text{ and } \mu_i = 0 \text{ for all but finitely many } i \right\}.
$$
Indeed, this follows from the fact that the weights of the modules $V^{\otimes m}\otimes V^{\otimes n}_*$ are all of this form.  For any such $X$, we define a $\Z$-grading $X=\bigoplus_{k\in \Z} X^k$ by letting $$X^k := \bigoplus_{\mu\in \h^*,\ |\mu|=k} X_\mu.$$  Notice that each $X^k$ is invariant under the action of $\g^0$. In particular, if $X$ simple, then $X=X^k$ for some $k\in \Z$, and in this case, we will write $|X|:=k$. For example, $|V_{\lambda, \mu}|=\sum \lambda_i-\sum \mu_j$.

\begin{prop}\label{prop:JH.Kac.g_0-module}
Let $X\in \bbTgl$ such that $|X|=k$. The module $K^\pm(X)$ considered as a $\g^0$-module admits a   decomposition  
\begin{equation}\label{eq: grading K}
K^+(X)=\bigoplus_{j\leq 0} T^j_+, \hspace{1cm}
K^-(X)=\bigoplus_{j\geq 0} T^j_-
\end{equation}
such that each $T^j_\pm$ is in $\bbTgl$   and each simple subquotient $Y$ of $T^j_\pm$ satisfies $|Y| =k+j$. The decomposition  
in (\ref{eq: grading K}) is a $\Z$-grading of the $\g$-module $K^\pm(X)$.
\end{prop}
\begin{proof}
Let us first prove the claim for $K^-(X)$. Since $\g^> = \bigoplus_{i>0} \g^i$ is graded, we can decompose the symmetric $\g^0$-module $S(\g^>)$ as $S(\g^>) = \bigoplus_{j>0} S^j$, where $S^j$ is the projection of $\bigoplus_{i_1+\cdots + i_n=j} \g^{i_1}\otimes\cdots \otimes \g^{i_n}$ onto $S(\g^>)$. So we have an isomorphism of $\g^0$-modules
	\[ 
K^-(X)\cong S(\g^>)\otimes X\cong \bigoplus_{j\geq 0} \left(S^j\otimes X\right).
	\]
For each $j\geq 0$, set $T^j_- := S^j\otimes X$, and observe that $T^j_-$ is in $\bbTgl$ since $S^j\in\bbTgl$. Notice now that any $\mu \in \Supp S^j$ satisfies $|\mu| = \sum \mu_i = j$. Thus, for $N\gg 0$, the element  $I_N := \sum_{i=1}^N \xi_i\partial_i$ acts on any weight vector of $T_-^j$ (and hence on any weight vector of a subquotient of $T^j_-$) via multiplication by $k+j$.  

Now for $K^+(X)$, we note that we have isomorphisms of $\g^0$-modules
    \[ 
K^+(X)\cong \Lambda(\g^<)\otimes X\cong \bigoplus_{j\geq 0} \Lambda^j(V_*)\otimes X,
    \]
where each component $T^{-j}_+ :=\Lambda^j(V_*)\otimes X$ is in $\bbTgl$, and hence has finite length over $\g^0$.   Moreover, for $N\gg 0$, the element $I_N$ acts on any weight vector of $(\Lambda^j(V_*)\otimes X)$  via multiplication by $k-j$, and the claim follows for $K^+(X)$.

Now it is clear that $\g^i T_{\pm}^j \subset T_{\pm}^{i+j}$, so the $\Z$-grading is compatible with the action of $\g$.
\end{proof}

\begin{cor}
If $X\in\bbTgl$ is simple, then the module $K:=K^\pm (X)$ admits a  $\g^0$-module filtration
	\[
K = K_0 \supset K_1 \supset K_2 \supset K_3  \cdots,
	\]
where $\bigcap_i K_i = 0$ and $K_{i-1}/K_i$ is simple for all $i>0$. Moreover, the multiplicity of any simple subquotient of $K$ is finite and  does not depend on the choice of a filtration.
\end{cor}
\begin{proof} 
For each $j\geq 0$, let $T^j = F_j^{0}\supset F_j^1\supset \cdots \supset F_j^{n_j}$ be the Jordan--H\"older series of $T^j$ provided by Proposition~\ref{prop:JH.Kac.g_0-module}. Now, observe that
\begin{multline*}
K_1 := F_1^1 \oplus \bigoplus_{j>1} T^j \supset K_2 := F_1^2 \oplus \bigoplus_{j>1} T^j \supset\cdots\supset K_{n_1}:= F_1^{n_1} \oplus \bigoplus_{j>1} T^j 
\supset \\
\supset K_{n_1+1} := \bigoplus_{j>1} T_j \supset K_{n_1+2} := F_2^1\oplus \bigoplus_{j>2} T_j \supset\cdots
\end{multline*}
yields the required filtration.
\end{proof}

\begin{prop} 
If $X\in\bbTgl$, then any simple $\g^0$-module in the Jordan--H\"older series of the induced module $K^\pm (X)$, considered as a $\g^0$-module, has finite multiplicity.
\end{prop}

\begin{proof}
If $X$ is simple, then any simple $\g^0$-subquotient $M$ of $K^\pm(X)$ appears in a unique component $T^j_\pm$ (see the proof of Proposition~\ref{prop:JH.Kac.g_0-module} for the definition of $T^j_\pm$) which is determined by $|M|$. Hence, $M$ must have finite multiplicity since $T^j_\pm$ has finite length.

Now we use induction on the length of $X$ in its Jordan--Hölder series. Let $Y$ be a maximal submodule of $X$ and consider the short exact sequence
    \[
0\to Y\to X\to X/Y\to 0
    \]
Applying the exact functor $K^\pm(-)$ on this sequence yields
    \[
0\to K^\pm(Y)\to K^\pm(X)\to K^\pm(X/Y)\to 0.
    \]
By induction, a simple $\g^0$-subquotient $M$ has finite multiplicity in  $ K^\pm(Y)$  and in $K^\pm(X/Y)$. Since the multiplicity of $M$ is additive, the result follows.
\end{proof}

Given $X\in \bbTgl$, we fix the $\Z$-grading of  $K^\pm(X)$ defined by applying Proposition~\ref{prop:JH.Kac.g_0-module} to each component $X^k$ of the decomposition $X=\bigoplus_{k\in \Z} X^k$ with $|X^k|=k$.  

\begin{theo}\label{prop:Kac.in.T_W} If $X\in \bbTgl$, then  $K^+(X)$ lies in $ \bbT_W^\geq$ and  $K^-(X)$ lies in  $\bbT_W^{\leq}$.
\end{theo}

\begin{proof}
Write $X=\bigoplus_{k\in \Z} X^k$ where $|X^k|=k$. Then, by  Proposition~\ref{prop:JH.Kac.g_0-module} applied to each $X^k$ we have that $K^\pm(X)$ satisfies (1) and (2) of Definition~\ref{def:T_W-Cat}.
So we just need to show that  $K^+(X)$ satisfies l.a.c. for $\g^{\geq}$ and that  $K^-(X)$ satisfies l.a.c. for $\g^{\leq}$.

Let $\fk$ be a subalgebra of $\g$, and let $\bU(\g)_{\ad \fk}$ denote $\bU(\g)$ viewed as a $\fk$-module via the adjoint action. Notice that $\g$ viewed as a $\g$-module (via the adjoint action) satisfies l.a.c. for $\g$ (and hence for $\g^{\geq}$). Thus $\bU(\g)_{\ad \g^\geq}$, $X$ (viewed as a $\g^{\geq}$-module with $\g^>\cdot X=0$) and $\bU(\g)_{\ad \g^\geq}\otimes_\C X$ also satisfy the l.a.c. for $\g^\geq$. Consider the surjective map of vector spaces $\bU(\g)_{\ad \g^\geq}\otimes_\C X  \xrightarrow{\phi} K^+(X)$ given by $\phi(u\otimes x)= u\otimes x$ for all $x\in X$, $u\in \bU(\g)$. This is a $\g^{\geq}$-module homomorphism since for all $g \in \g^{\geq}$ and $x\in X$, we have 
    \[
\phi(g(u\otimes x))=\phi\left([x,u]\otimes x +(-1)^{\bar{u}\bar{g}} u\otimes gx\right) = [x,u]\otimes x +(-1)^{\bar{u}\bar{g}} u\otimes gx
    \]
and $g\phi(u\otimes x)=gu\otimes x=[g,u]\otimes x+(-1)^{\bar{u}\bar{g}}u\otimes gx$. Thus the result follows for $K^+(X)$. The proof for $K^-(X)$ is similar.
\end{proof}

\begin{rem}\label{rem: K lac}
The induced $\g$-module $K^+(X)$ is not in $\bbT_W^{\leq}$ or $\bbT_W$, since it does not satisfy the l.a.c. for $\g^\leq$. Similarly, $K^-(X)$ is not in $\bbT_W^{\geq}$ or $\bbT_W$, since it does not satisfy the l.a.c. for $\g^\geq$.
\end{rem} 

%%%%%%%%%%%%%%%%
%
\section{Simple modules in $\bbT_W$}
%
%%%%%%%%%%%%%%%%

In this section we prove that the simple modules of $\bbT_W$ can be parameterized by two partitions, which we denote by $L^-_{\lambda,\mu}$, and that they are highest weight modules with respect to the Borel subalgebra $\fb(\prec)^{\min}$.

Given a Borel subalgebra $\fb^0$   of $\g^0=\fgl(\infty)$, we define  Borel subalgebras for $\g$  by
$$\fb^{\max}:=\fb^0\oplus\g^>, \hspace{1cm} \fb^{\min}:=\fb^0\oplus\g^<.$$ Similarly, for a Borel subalgebra $\fb_n^0$ of  $\g_n^0=\fgl(n)$, we define Borel subalgebras of $\g_n$ by $\fb_n^{\max}:=\fb_n^0\oplus\g_n^>$ and $\fb_n^{\min}:=\fb_n^0\oplus \g_n^<$.

For any weight $\gamma\in \h^*$, we denote by $V_{\fb^0}(\gamma)$   the simple $\fb^0$-highest weight $\g^0$-module with highest weight $\gamma\in\h^*$.  Similarly, $V_{\fb^0_n}(\gamma)$  denotes the simple $\fb^0_n$-highest weight $\g^0_n$-module with highest weight $\gamma\in\h_n^*$, where $\h_n:=\h\cap\g_n^0$.

For a Borel subalgebra $\fb \in \{\fb^{\max}, \fb^{\min}\}$ of $\g$, let $\C_\gamma$ denote the one-dimensional representation of $\fb$ defined by $\gamma\in\h^*$. The induced module $M_{\fb}(\gamma):=U(\g)\otimes_{U(\fb)} \C_{\gamma}$ is a $\fb$-highest weight $\g$-module. We let $L_{\fb}(\gamma)$ denote the unique simple quotient of $M_{\fb}(\gamma)$.

Recall the Borel subalgebra $\fb(\prec)^0$ of $\g^0$ corresponding to the linear order $\prec$ on $I$ given in (\ref{prec def}).
We let $\fb(<)^0 :=\h\oplus \n(<)$  denote the standard Borel corresponding to the natural order on $I$.
In what follows we set 
	\[
\fb(\prec)_n^0:=\fb(\prec)^0\cap \fgl(n), \text{ and } \fb(<)_n^0:=\fb(<)^0\cap \fgl(n).
	\]

\begin{rem}\label{rem Ser W}
Serganova studied  induced  $W(n)$-modules in \cite{S05}, where she proved that the $\g_n$-module $K_n^+( V_{\fb(<)_n^0}(\nu))$ is simple if and only if the weight $\nu$ is typical \cite[Theorem~6.3]{S05}, and has length $2$ when $\nu$ is atypical \cite[Corollary~7.6]{S05}.  By Theorem~5.3 in \cite{S05}, the atypical weights for the standard Borel subalgebra $\fb(<)_n^{\max}$ are of the form $a\varepsilon_i+\varepsilon_{i+1}+\cdots +\varepsilon_n$, for some $a\in\C$.
\end{rem}

Recall that $V_{\lambda, \mu}$ denotes the simple module in $\bbTgl$ corresponding to partitions $\lambda,\mu$ (see Theorem~\ref{th:PSt}), that the induced $\g$-module $K^+(V_{\lambda, \mu})$ is a  $\fb(\prec)^{\max} := \fb(\prec)^0\oplus \g^{>}$ highest weight module with highest weight 
$(\lambda,\mu)$, and  $K^-(V_{\lambda, \mu})$ is a  $\fb(\prec)^{\min}:=\fb(\prec)^0\oplus \g^{<}$ highest weight module with highest weight 
$(\lambda,\mu)$.

By our convention, $(0)$ represents the empty partition $\emptyset$.

\begin{theo}\label{prop:K+.simple}
The $W(\infty)$-module $K^+(V_{\lambda, \mu})$ is simple (and in fact, locally simple) if and only if  $(\lambda,\mu)\neq (\emptyset,(\mu_1))$, that is, if and only if $V_{\lambda, \mu}\not\cong S^k(V_*)$ for some $k\in\Z_{\geq 0}$.
\end{theo}
\begin{proof}
Note that any simple $X\in \bbTgl$ is such that $X\cong  V_{\lambda, \mu}$ for some partitions $\lambda$, $\mu$ \cite{PSt11,DPS16}. In particular, $X\cong \varinjlim X_n$, where for $n\gg 0$, the $\g_n^0$-module $X_n$ is the simple $\fb(\prec)_n^0$-highest weight module 
$$
V_{\fb(\prec)_n^0}\left(\sum_{i\in \Zp}\lambda_i \varepsilon_{2i-1} -\sum_{i\in  \Zp}\mu_i\varepsilon_{2i}\right).
$$
In particular, using the Weyl group of $\g_n^0$, we conclude that
	\[
X_n \cong V_{\fb(<)_n^0}\left(\lambda_1\varepsilon_1 +\cdots + \lambda_k\varepsilon_k - \mu_l\varepsilon_{n-l}-\cdots -\mu_1\varepsilon_n\right).
	\]

It follows from Remark~\ref{rem Ser W} that  $K_n^+(X_n)$  is simple for $n > k+l$ if and only if its highest weight with respect to  $\fb(<)_n^{\max} $ is
not equal to  $-\mu_1\varepsilon_n$, that is, if $(\lambda,\mu)\neq (\emptyset,(\mu_1))$. Since $K^+(V_{\lambda, \mu})\cong \varinjlim K_n^+(X_n)$, we conclude that $K^+(V_{\lambda, \mu})$ is locally simple  when  $(\lambda,\mu)\neq (\emptyset,(\mu_1))$. 
For the second claim, note that $V_{\emptyset,(k)}= S^k(V_*)$.

If  $(\lambda,\mu)= (\emptyset,(k))$,  then  $V_{\emptyset,(k)}= S^k(V_*)$. The vector $1\otimes \partial_2^k$ is the $\fb(\prec)^{\max}$-highest weight vector of $K^+(S^k(V_*))$. We claim that the weight vector $\partial_2\otimes \partial_2^k\in K^+(S^k(V_*))$ is  $\fb(\prec)^{\max}$-primitive and generates a proper submodule. Indeed,  for any homogeneous element $ x:=\underline{\xi}^{\underline{a}} \partial_j\in \g^{\geq}$ we have 
$$
x(\partial_2\otimes \partial_2^k)=[x,\partial_2]\otimes \partial_2^k+ (-1)^{\bar x }\partial_2\otimes x\partial_2^k,
$$
and since $\partial_2^k$ is $\fb(\prec)^{\max}$-primitive, we just need to show that $[x,\partial_2]\otimes \partial_2^k$ is zero when $x\in\n(\prec)^0$ and $x\in \g^1$.
By (\ref{bracket form}), $[\underline{\xi}^{\underline{a}}\partial_j,\partial_2]$ is zero unless  $\underline{a}(2)=1$, and hence equals zero when  $x\in\n(\prec)^0$. Now consider $x_k:=\underline{\xi}^k \underline{\xi}^2 \partial_k\in\g_{\varepsilon_2}$, when $k\in\Z_{>0}$ and $k\neq 2$. Since $(\ad \partial_2)^2=0$ by the Jacobi identity, we get $h_k:=[x_k,\partial_2]\in\h_{-\varepsilon_2}$ and $\varepsilon_2(h_k)=0$. 
Hence,
$ [x_k,\partial_2]\otimes\partial_2^k=1\otimes h_k\partial_2^k=0
$, and the claim follows.
\end{proof} 

\begin{rem}
We can define a $\g$-module homomorphism $K^+(S^{k+1}(V_*))\to K^+(S^k
(V_*))$ by mapping $1\otimes\partial_2^{k+1}\mapsto \partial_2\otimes \partial_2^k$. It was proved for the finite-dimensional case in \cite{S05}, that the $\g_n$-module $K^+(S^{k}(V_n)^*)$ has length 2 and the analogous maps give a resolution for $K^+(S^k(V_n)^*)$.
\end{rem}

In contrast with Theorem~\ref{prop:K+.simple} we have the following. 

\begin{prop}\label{prop:K^-not.simple}
If $X$ is a $\fb(\prec)^0$-highest weight $\g^0$-module, then $K^-(X)$ is not simple. In particular, $K^-(V_{\lambda, \mu})$ is never simple for any simple $V_{\lambda, \mu}\in \bbTgl$.
\end{prop}
\begin{proof}
Let $v\in X$ be a nonzero $\fb(\prec)^0$-highest weight vector, and let $x\in \g_{\varepsilon_1 + \varepsilon_3 - \varepsilon_2}\subseteq \g^1$ be a nonzero root vector. Since $\g^1\cong \Lambda^2(V)\otimes V_*$ as $\g^0$-modules, $x$ is a $\fb(\prec)^0$-highest weight vector of $\g^1$. We claim that $w:=xv\in K^-(X)$ is a $\fb(\prec)^{\min}$-highest weight vector. Indeed, since $x$ and $w$ are $\fb(\prec)^0$-highest weight vectors, it is clear that $w$ is a $\fb(\prec)^0$-highest weight vector. Now let $y_i\in \g_{-\varepsilon_i}\subseteq \g^{-1}$ be a nonzero root vector, and observe that either $[y_i, x]\in \n(\prec)^0$ or $[y_i, x]=0$. In any case, we get $y_iw=0$, and thus $\g^{-1}w=0$. This proves the claim, and the result.
\end{proof}

\begin{cor}
If $X\in \bbTgl$, then $K^-(X)$ is not simple.
\end{cor}
\begin{proof}
If $X$ is not simple, then $K^-(X)$ is not simple. If $X$ is simple then $X\cong V_{\lambda,\mu}$ for some partitions $\lambda$, $\mu$, so the claim follows from Proposition~\ref{prop:K^-not.simple}.
\end{proof}

For a subalgebra $\fk$ of $\g$ and a $\g$-module $M$, we define $M^{\fk}=\{m\in M\mid \fk\cdot m=0\}$. Recall that $\g^<=\g^{-1}$. We define the functor $\Psi: \bbT_W^\leq\to \tbbTgl$ by
	\[
\Psi(M) := M^{\g^<}.
	\]
	
\begin{prop}\label{prop:inv.non-zero}
If $M\in \bbT_W^\leq$, then $\Psi(M)\neq 0$. If $M\in \bbT_W^\leq$ is simple, then
\begin{enumerate}
\item $\Psi(M)$ is a simple $\g^0$-module, 
\item $\Psi(M)$ is in $\bbTgl$, and 
\item $M\cong L^- (\Psi(M))$.
\end{enumerate}
\end{prop}
\begin{proof}
For a nonzero vector $m\in M$ there exists $n$ such that $(\ft_{n}\cap\g^{\leq}) \subseteq \Ann (m)$. Choose $m'$ to be the longest non-zero element of the form $\partial_{i_1}\cdots \partial_{i_k}\cdot m$ such that $1\leq i_1< \cdots < i_k\leq n$. Then $\partial_j \cdot m'=0$ for every $j\leq n$. Moreover,  $\partial_j\cdot m'=0$ for all $j>n$, since $\partial_j\in(\ft_{n}\cap\g^{\leq}) \subseteq \Ann (m)$ and $\partial_i\partial_j=-\partial_j\partial_i$ for any $i,j\in I$. This shows that $m'\in M^{\g^<}$ and the first part is proved. 
    
Assume now that $M\in \bbT_W$ is simple. The first part and the simplicity of $M$ implies that there exist $n\in \Z$ such that $M=\bigoplus_{k\geq n}M^k$. Indeed, if $(M^{\g^{<}})^n$ is nonzero, then $M=\bU(\g) (M^{\g^{<}})^n=\bU(\g^{>}) (M^{\g^{<}})^n$, which determines $n$. Moreover,  we get $M^{\g^<}= M^n$. 
    
Now, if $N$ is a proper $\g^0$-submodule of $M^{\g^<}$, then $\bU(\g) N=N\oplus \bU(\g^{>})^+N$ and hence $M^{\g^<}\cap \bU(\g) N=N$. This implies that $\bU(\g) N$ is a proper $\g$-submodule of $M$, contradicting the irreducibility of $M$. Since $M^{\g^<}$ is a simple $\g^0$-submodule of $M$, and $M|_{\g^0}\in \tbbTgl$, we get $M^{\g^<}\in \bbTgl$.  Finally, we have $M\cong L^-(M^{\g^<})$ since $M=\bU(\g^{>}) M^{\g^{<}}$.
\end{proof}

\begin{rem}
The proof of Proposition~\ref{prop:inv.non-zero} does not work if we replace $\g^<$ by $\g^>$, since unlike $\bU(\g^<)$, the algebra $\bU(\g^>)$ is not a Grassmann algebra. In fact, $\Lambda(V)$, $\Lambda(V)_+$ and $\g$ are examples of modules in $\bbT_W$ (and hence in $\bbT_W^\geq$) for which $M^{\g^>}=0$. Actually, we cannot have a simple module $M\in \bbT_W$ with $M^{\g^>}\neq 0$. Indeed, if this were the case, then we could use the same arguments as those of Proposition~\ref{prop:inv.non-zero} to conclude that $M^{\g^>}\in \tbbTgl$, and hence that $M\cong K^+(M^{\g^>})$ (by Theorem~\ref{prop:K+.simple}). But $K^+(M^{\g^>})\notin \bbT_W$, since it does not satisfy the l.a.c. for $\g^\leq$.
\end{rem}

\begin{example}\label{ex:nat.con.in.T_W}
Note that $L^-(V)$  is isomorphic to $\Lambda(V)_+:=\Lambda(V)/\C$, since $(\Lambda(V)_+)^{\g^<} = V$. And, $L^-(V_*)$ is isomorphic to the adjoint module of $\g$, since it satisfies $\g^{\g^<} = \g^{-1}\cong_{\g^0} V_*$. A direct computation shows that the $\g$-modules  $\Lambda(V)_+$ and $\g$ are both objects in $\bbT_W$ (see Example~\ref{ex:TW}), and they have $\fb(\prec)^{\min}$-highest weights $\varepsilon_1$ and  $-\varepsilon_2$, respectively. \hfill  $\bigcirc$ 
\end{example}

We set $L^-_{\lambda,\mu}:=L^-(V_{\lambda,\mu})$. Then $L^-_{\lambda,\mu}$ is a highest weight $\g$-module with respect to the Borel subalgebra
$\fb(\prec)^{\min}$ and has highest weight given by (\ref{eq hw}). It follows from Proposition~\ref{prop:inv.non-zero} that $ \Psi (L^-_{\lambda,\mu})\cong V_{\lambda,\mu}$. The next result describes the simple objects of $\bbT_W^{\leq}$ in terms of partitions.

\begin{theo}\label{theo: simples Lminus}
The $\g$-module $L^-_{\lambda,\mu}$ is in the category $\bbT_W^{\leq}$ for any partitions $\lambda$ and $\mu$. Moreover, if $M\in \bbT_W^{\leq}$ is simple, then there exist partitions $\lambda$ and $\mu$ for which $M\cong L^-_{\lambda, \mu}$.
\end{theo}
\begin{proof}
By Theorem~\ref{prop:Kac.in.T_W}, $K^-(V_{\lambda,\mu})\in\bbT_W^\leq$, so  $L^-_{\lambda,\mu}\in\bbT_W^\leq$ since $\bbT_W^\leq$ is closed under taking quotients.
Now by  Proposition~\ref{prop:inv.non-zero}, for any simple $\g$-module $M$, we have that $\Psi(M)$ is a simple object of $\bbTgl$, and hence $\Psi(M)\cong V_{\lambda,\mu}$ for some partitions $\lambda,\mu$. Moreover,  $M\cong L^- (\Psi(M))\cong L^- (V_{\lambda,\mu})=L^-_{\lambda, \mu}$.
\end{proof}

\begin{cor}\label{cor:simple.in.T_W}
Every simple module of $\bbT_W^{\leq}$ (and hence of $\bbT_W$) is a highest weight module with respect to the Borel subalgebra $\fb(\prec)^{\min}$.
\end{cor}

\begin{prop}\label{prop:simple.in.T_W}
For any pair of partitions $\lambda$ and $\mu$, there exist $m,n\in \Zp$ such that the $\g$-module $L^-_{\lambda,\mu}$ is a subquotient of $L^-(V)^{\otimes m}\otimes L^-(V_*)^{\otimes n}$. In particular, $L^-_{\lambda,\mu}$ lies in $\bbT_W$.
\end{prop}
\begin{proof}
Let $m=|\lambda|$ and $n=|\mu|$. Then $V^{\otimes m}\otimes V_*^{\otimes n}$ admits a $\fb(\prec)^0$-highest weight vector $v\neq 0$ of highest weight $\sum_{i\in \Zp}\lambda_i \varepsilon_{2i-1} -\sum_{i\in \Zp}\mu_i\varepsilon_{2i}\in \h^*$. The obvious embedding of $\g^0$-modules of $V$ into $K^-(V)$ gives rise to an embedding of $\g^0$-modules of $V$ into $L^-(V)^{\g^<}$, since $V$ is simple. Therefore we have an embedding $\g^0$-modules
	\[
V^{\otimes m}\otimes V_*^{\otimes n} \hookrightarrow \left(L^-(V)^{\otimes m}\otimes L^-(V_*)^{\otimes n}\right)^{\g^<},
	\]
and we may assume that $v\in \left(L^-(V)^{\otimes m}\otimes L^-(V_*)^{\otimes n}\right)^{\g^<}$. Then $v$ is a $\fb(\prec)^{\min}$-highest weight vector, and $L^-_{\lambda,\mu}$ is a subquotient of $L^-(V)^{\otimes m}\otimes L^-(V_*)^{\otimes n}$. Since $L^-(V)$ and $L^-(V_*)$ are in $\bbT_W$ (see Example~\ref{ex:nat.con.in.T_W}), the result follows now from the fact that $\bbT_W$ is closed under tensor products and subquotients.
\end{proof}

\begin{rem}
Proposition~\ref{prop:simple.in.T_W} shows that the $\g$-modules $L^-(V)$ and $L^-(V_*)$ play a similar role for the category $\bbT_W$ as that of the $\g^0$-modules $V$ and $V_*$ for the category $\bbTgl$. In the category $\bbTgl$, all simple objects can be realized as subquotients of some $V^{\otimes m}\otimes V_*^{\otimes n}$.
\end{rem}

We obtain the following corollary from Theorem~\ref{theo: simples Lminus} and \ref{prop:simple.in.T_W}.

\begin{cor}\label{cor: coincide}
The simple objects of $\bbT_W$ and $\bbT_W^{\leq}$ coincide.
\end{cor}

\begin{rem}
Note that Corollary~\ref{cor: coincide} does not hold if we replace $\bbT_W^{\leq}$ with $\bbT_W^{\geq}$.
Indeed,  by Theorem~\ref{prop:Kac.in.T_W} and Theorem~\ref{prop:K+.simple},  the induced $\g$-module $K^+(V_{\lambda, \mu})$ is simple when  $(\lambda,\mu)\neq (\emptyset,(\mu_1))$; however, we have seen in Remark~\ref{rem: K lac} that $K^+(V_{\lambda, \mu})$ is never in $\bbT_W$.
\end{rem}

%%%%%%%%%%%%%%%%
%
\section{Realization of simple modules as tensor fields}\label{sex:tensor_modules}
%
%%%%%%%%%%%%%%%%

In this section, we prove that the simple module  $L^-_{\lambda,\mu}$ can be realized as a module of tensor fields. This is a natural generalization of the classical work of Bernstein and Leites for the finite-dimensional Cartan type Lie superalgebra $W(n)$  \cite{BL81}.

%This extends results from the finite-dimensional case to our setting. 

Let $X$ be a $\fgl(\infty)$-module.  We define a $\g$-module structure on the graded vector space 
    \[
\cT(X) := \Lambda (\infty)\otimes X = \bigoplus_{k\geq 0} \Lambda(\infty)^k\otimes X
    \]
by linearly extending the following action:
	\[
\underline{\xi}^{\underline{e}} \partial_j \cdot (fv) = \underline{\xi}^{\underline{e}} \partial_j(f)v +  (-1)^{p(\underline{\xi}^{\underline{e}} \partial_j)p(f)}\sum_{i\in I} \partial_i(\underline{\xi}^{\underline{e}}) f E_{i,j}v,
	\]
where $f\in \Lambda (\infty)$ and $v\in X$. This action is well defined since $\partial_i(\underline{\xi}^{\underline{e}})$ is nonzero only for finitely many $i\in I$. We call the module $\cT(X)$ the module of tensor fields associated to $X$.

\begin{lem}\label{em:TX_in_tT}
If $X\in \tbbTgl$, then $\cT(X)$ lies in $\bbT_W$. \end{lem} 
\begin{proof}
Now $\cT(X)$ is $\Z$-graded by definition, and since $X$, $\Lambda(\infty)$ are integrable over $\g^0$, we have that $\cT(X)$ is also integrable over $\g^0$. Finally, $\cT(X)$ satisfies the l.a.c. for $\g$, since the action of $\g$ on $\Lambda(\infty)$ and the action of $\fgl(\infty)$ on $X$ both satisfy the l.a.c.. 
\end{proof}

\begin{example}
If $V_*$ is the conatural $\fgl(\infty)$-module, then the $\g$-module $\cT(V_*) = \Lambda(\infty)\otimes V_*$ is isomorphic to the adjoint module $ \g$ (see Section~\ref{sec: W mod}). If $\C$ is the trivial $\fgl(\infty)$-module, then $\cT(\C)\cong \Lambda(\infty)$ is the natural module. \hfill $\bigcirc$ 
\end{example}

Let's consider $X$ as a $\g^{\geq}$-module by declaring $\g^{>}\cdot X=0$. Then we define
	\[
\Coind_{\g^{\geq}}^\g (X) := \Hom_{\g^{\geq}}(\bU(\g), X).
	\]
An easy computation shows that if $t\in X\cong \C\otimes X\subseteq \cT(X)$, then $\partial_i\cdot t = 0$ for all $i\in I$. For an arbitrary $t = \sum f_iv_i \in \cT(X)$, we set
	\[
t(0) := \sum f_i(0)v_i\in X.
	\]

Now we have a natural embedding of $\g$-modules $\varphi: \cT(X)\hookrightarrow\Coind_{\g^{\geq}}^\g (X)$, where  for $t\in\cT(X)$ and $u\in \bU(\g)$,
	\[
\varphi(t)(u) := (-1)^{p(t)p(u)}(u\cdot t)(0).
	\]

Next we consider the finite-dimensional case. Suppose we have inclusions of $\g_n^0$-modules $X_n \subseteq X_{n+1}$ for every $n\in \Zp$, and consider the $\g^0$-module $X = \bigcup X_n$. We set $\cT(X_n):= \Lambda(n)\otimes X_n$. Then we have $\cT(X) = \bigcup \cT(X_n)$, where we regard $\cT(X_n)\subseteq \cT(X_{n+1})$ in the obvious way. The restriction $\varphi_n:=\varphi |_{\cT(X_n)}$ gives  embeddings of $\cT(X_n)$ into $\Hom_{\g_n^\geq}(\bU(\g_n), X_n)$. Moreover we have the following statement from \cite{BL81}.

\begin{lem}\label{lem:tens.fields=coind}
$\cT(X_n)\cong \Coind_{\g_n^{\geq}}^{\g_n} (X_n)$, where the isomorphism is $\varphi_n$.
\end{lem}
\begin{proof}
Since $\varphi_n$ is an embedding, we just have to check that the dimensions match. This follows from the following
\begin{align*}
\cT(X_n) & := \Lambda(\xi_1,\ldots, \xi_n) \otimes_\C X_n \\
& \cong \Lambda(\partial_1,\ldots, \partial_n)^* \otimes_\C X_n \\
& \cong \Hom_{\C}(\Lambda(\partial_1, \ldots, \partial_n) \, , X_n) \\
& \cong \Hom_{\C}(\Lambda(\partial_1, \ldots, \partial_n) \, , \Hom_{\g_n^\geq }(\bU(\g_n^\geq) \, ,  X_n)) \\
& \cong \Hom_{\g_n^{\geq}}(\Lambda(\partial_1,\ldots, \partial_i)\otimes_\C \bU(\g_n^\geq) \, , X_n) \\
& \cong \Hom_{\g_n^{\geq}}(\bU(\g_n) \, , X_n). \qedhere
\end{align*}
\end{proof}

\begin{rem}
In the infinite-dimensional case, the second isomorphism displayed above is just a natural embedding from $\Lambda(\partial)^* \otimes_\C X$ to $\Hom_{\C}(\Lambda(\partial) \, , X)$.
\end{rem}

Consider the submodule $\bigcup \im\varphi_n$ of $\Coind^{\g}_{\g^\geq}(X)$. It follows from Lemma~\ref{lem:tens.fields=coind} that\break $\bigcup \im\varphi_n\cong \bigcup \cT(X_i)$ is a submodule of $\cT(X)$. But since every element of $\cT(X)$ lies in some $\cT(X_n)$ for $n\gg 0$, we have the following:
    \[
\cT(X) = \bigcup \cT(X_n)\cong\bigcup \im\varphi_n = \bigcup \Coind^{\g_n}_{\g_n^\geq}(X_n)\subseteq \Coind^{\g}_{\g^\geq}(X).
    \]

For any finite-dimensional $\g_n^0$-module $X_n$, we have
\begin{align*}
\cT(X_n^*) & \cong  \Hom_{\g_n^{\geq}} (\bU(\g_n) \, ,  \Hom_\C(X_n \, , \C)) \\
&  \cong \Hom_\C (\bU(\g_n)\otimes_{\bU(\g_n^\geq)} X_n \, , \C)  = (\Ind^{\g_n}_{\g_n^{\geq}} X_n)^* = K_n^+ (X_n)^*.
\end{align*}
In particular, $\cT(X_n)\cong K_n^+(X_n^*)^*$. Using this fact one can prove the following result.

\begin{prop}\label{prop:TV_simple}
If $(\lambda,\mu)\neq ((\lambda_1), \emptyset)$, then $\cT(V_{\lambda, \mu})$ is a locally simple $\g$-module, and hence simple.
\end{prop}
\begin{proof}
Let $\lambda = (\lambda_1\geq \cdots \geq \lambda_k)$ and $\mu = (\mu_1\geq \cdots \geq \mu_l)$. Recall that the simple $\g^0$-module $X:=V_{\lambda, \mu}$ has a decomposition $X = \bigcup X_n$, where, for $n\gg 0$, $X_n$ is the $\g_n^0$-module
	\[
X_n := V_{\fb(<)_n^0}(\lambda_1\varepsilon_1 +\cdots + \lambda_k\varepsilon_k - \mu_l\varepsilon_{n-l}-\cdots -\mu_1\varepsilon_n).
	\]
Now $\cT(X_n)\cong K_n^+(X_n^*)^*$, where
	\[
X_n^* \cong V_{\fb(<)_n^0}(\mu_1\varepsilon_1 +\cdots + \mu_l\varepsilon_l - \lambda_k\varepsilon_{n-k}-\cdots -\lambda_1\varepsilon_n).
	\]
It follows from Remark~\ref{rem Ser W} that the highest weight $\mu_1\varepsilon_1 +\cdots + \mu_l\varepsilon_l - \lambda_k\varepsilon_{n-k}-\cdots -\lambda_1\varepsilon_n$ of $X_n^*$ is atypical if and only if $(\lambda,\mu) = ((\lambda_1), \emptyset)$. Hence,  $\cT(X_n)\cong K_n^+(X_n^*)^*$ is simple if and only if $(\lambda,\mu)\neq ((\lambda_1), \emptyset)$. Now the claim follows from the fact that $\cT(X)\cong \bigcup \cT(X_n)$. 
\end{proof}

The following theorem gives a realization of the simple modules in $\bbT_W$ as tensor fields.

\begin{theo}\label{thm:simples_are_tensor_modules}
For any $(\lambda,\mu)$ we have that $L_{\lambda, \mu}^-$ is isomorphic to a submodule of $\cT(V_{\lambda, \mu})$. Moreover, if  $(\lambda,\mu)\neq ((\lambda_1), \emptyset)$, then $ L_{\lambda, \mu}^-\cong \cT(V_{\lambda, \mu})$.
\end{theo}
\begin{proof} 
For $n > k+l$, we denote by $V_{\lambda,\mu,n}$ the  $\g_n^0$-module $V_{\fb(\prec)_n^0}(\nu)$ with  highest weight  $\nu := \sum_{i\in \Zp}\lambda_i \varepsilon_{2i-1} -\sum_{i\in  \Zp}\mu_i\varepsilon_{2i}$. We notice that for all $n>k+l$ we have a short exact sequence of $\g_n$-modules
    \[
0\to Q_n(V_{\lambda, \mu,n}^*)\to K_n^+(V_{\lambda, \mu,n}^*)\to L_n^+(V_{\lambda, \mu,n}^*)\to 0
    \]
which yields the following short exact sequence of $\g_n$-modules
    \[
0\to L_n^+(V_{\lambda, \mu,n}^*)^* \to K_n^+(V_{\lambda, \mu,n}^*)^*\cong \cT(V_{\lambda, \mu,n}) \to Q_n(V_{\lambda, \mu,n}^*)^* \to 0.
    \]
In particular, if $v$ denotes the $\fb(\prec)_n^0$-highest weight of $V_{\lambda, \mu,n}^*$, then $L_n^+(V_{\lambda, \mu,n}^*)^*$ is the submodule of $\cT(V_{\lambda, \mu,n})$ generated by  $v$. Since the weight of $v$ in $\cT(V_{\lambda, \mu,n})\subset \cT(V_{\lambda, \mu})$ is $\nu$ and $\cT(V_{\lambda, \mu,n})_{\nu} = (V_{\lambda, \mu,n})_{\nu}=(V_{\lambda, \mu})_{\nu}$, we conclude that $v$ is a $\fb(\prec)_n^0$-highest weight of $V_{\lambda, \mu,n}$ and hence  a $\fb\left(\prec\right)_n^{\min}$-highest weight vector in $\cT(V_{\lambda, \mu,n})$. Thus we have isomorphisms of $\g_n$-modules $L_n^+(V_{\lambda, \mu,n}^*)^*\cong \bU(\g_n)v\cong L_n^-(V_{\lambda, \mu,n})$, for all $n\gg 0$. This implies that $\bU(\g)v$ is locally simple and hence isomorphic to $L^-_{\lambda, \mu}$. Finally, the latter statement follows from Proposition~\ref{prop:TV_simple}.
\end{proof}

\section{Injective modules}\label{Section 7}
In this section we prove that the category $\bbT_W$ has enough injective objects. Moreover, for each simple module $L_{\lambda,\mu}^-$ in $\bbT_W$, we use Theorem~\ref{thm:simples_are_tensor_modules} to provide an explicit injective module in $\bbT_W$ that contains $L_{\lambda,\mu}^-$. Let $\g\Mod$ (resp. $\g^0\Mod$) be the category of all $\g$-modules (resp. $\g^0\Mod$), $\Int_{\g^0}$ be full subcategory of $\g^0\Mod$ consisting of integrable $\g^0$-modules, and $\Int_{\g,\g^0}$ be the full subcategory of $\g\Mod$ consisting of $\g^0$-integrable modules. Define the functors $\Gamma_{\g,\g^0}:\g\Mod\to \Int_{\g,\g^0}$ where $\Gamma_{\g,\g^0}(M) = \{m\in M\mid \dim \Span\{g^i m\mid i\geq 0\}<\infty,\ \forall g\in \g^0\}$ is the largest submodule of $M$ that is $\g^0$-integrable,  and $\Theta:\Int_{\g,\g^0} \to\bbT_W$ where $\Theta(M) = \bigcup_{n>0} M^{\ft_n}$ is the largest submodule of $M$ that satisfies the l.a.c. for $\g$. Let $\Gamma:\g\Mod\to \bbT_W$ denote the composition $\Theta \circ \Gamma_{\g,\g^0}$.

Both functors $\Gamma_{\g,\g^0}$ and $\Theta$ are left exact and are right adjoint to the respective inclusion functors $\Int_{\g,\g^0}\subset \g\Mod$ and $\bbT_W\subset \Int_{\g,\g^0}$. Now we have the following result:

\begin{prop}\label{prop:enough_inj}
If $I\in \g\Mod$ is injective, then $\Gamma(I)$ is injective in $\bbT_W$. Moreover, the category $\bbT_W$ has enough injectives.
\end{prop}
\begin{proof}
The former statement follows from the fact that $\Gamma$ is right adjoint to the inclusion functor $\bbT_W\subset \g\Mod$. To prove the latter statement, take $M\in \bbT_W$ and notice that if $I_M$ is the  injective hull of $M$ in $\g\Mod$, then the natural isomorphism
    \[
\Hom_\g (M \, ,I_M)\cong \Hom_\g (M \, ,\Gamma(I_M))
    \]
implies that the inclusion in $\Hom_\g (M \, ,I_M)$ gives an inclusion in $\Hom_\g (M \, ,\Gamma(I_M))$.
\end{proof}

\begin{cor}
If $I\in \g^0\Mod$ is injective, then $\Gamma (\Coind_{\g^0}^{\g}(I))\in \bbT_W$ is injective.
\end{cor}
\begin{proof}
This is because the functor $\Coind_{\g^0}^{\g}(-):\g^0\Mod\to \g\Mod$ sends injectives to injectives \cite[Corollary~6.4]{Kna88}.
\end{proof}

\begin{prop}
If $I\in \Int_{\g^0}$ is injective, then $\Gamma_{\g,\g^0} (\Coind_{\g^0}^{\g}(I))$ is injective in $\Int_{\g,\g^0}$.
\end{prop}
\begin{proof}
Let $M\in \Int_{\g^0}$. The adjunction between restriction and coinduction gives us that
    \[
\Ext_\g^1(M, \Coind_{\g^0}^{\g}(I))\cong \Ext_{\g^0}^1(M,I)=\Ext_{\Int_{\g^0}}^1(M,I)=0,
    \]
where we are using that $\Int_{\g^0}$ is closed under taking extensions \cite{PS11}. Now consider the short exact sequence
    \[
0\to \Gamma_{\g,\g^0} (\Coind_{\g^0}^{\g}(I)) \to \Coind_{\g^0}^{\g}(I) \to Q := \Coind_{\g^0}^{\g}(I)/ \Gamma_{\g,\g^0} (\Coind_{\g^0}^{\g}(I))\to 0
    \]
and the corresponding long exact sequence 
\begin{multline*}
0\to \Hom_\g(M,\Gamma_{\g,\g^0} (\Coind_{\g^0}^{\g}(I))) \to \Hom_\g(M,\Coind_{\g^0}^{\g}(I)) \to  \\  
\to \Hom_\g(M,Q) \to \Ext_\g^1(M,\Gamma_{\g,\g^0} (\Coind_{\g^0}^{\g}(I))) \to \Ext_\g^1(M,\Coind_{\g^0}^{\g}(I))=0
\end{multline*}
Since $M\in \Int_{\g^0}$, then $\Hom_\g(M,Q)=0$. Thus $\Ext_\g^1(M,\Gamma_{\g,\g^0} (\Coind_{\g^0}^{\g}(I)))=0$, which proves the statement.
\end{proof}

Let $\Gamma_{\g^0}:\g^0\Mod\to\Int_{\g^0}$ be the functor defined by $\Gamma_{\g^0}(M) = \{m\in M\mid \dim \Span\{g^i m\mid i\geq 0\}<\infty,\ \forall g\in \g^0\}$.

\begin{cor}
For every $M\in \Int_{\g^0}$, we have that  $\Gamma_{\g,\g^0} (\Coind_{\g^0}^{\g}(\Gamma_{\g^0}(M^*)))$ is injective in $\Int_{\g,\g^0}$.
\end{cor}
\begin{proof}
It follows from \cite[Proposition~3.1]{PS11}.
\end{proof}

\begin{cor}\label{cor:constructing_inj}
If $I\in \Int_{\g^0}$ is injective, then $\Gamma (\Coind_{\g^0}^{\g}(I))$ is injective in $\bbT_W$. In particular, for every $I\in \Int_{\g^0}$ we have that $\Gamma (\Coind_{\g^0}^{\g}(\Gamma_{\g^0}(I^*)))$ is injective in $\bbT_W$.
\end{cor}
\begin{proof}
It follows from the fact that  $\Theta$ is right adjoint to the respective inclusion functor $\bbT_W\subset \Int_{\g,\g^0}$.
\end{proof}

\begin{prop}\label{prop:injective_for_simple}
Each simple module $L^-_{\lambda,\mu}$ of $\bbT_W$ is isomorphic to a submodule of the injective module $\Gamma \left(\Coind_{\g^0}^{\g}(((V_{\lambda,\mu})_{*})^*)\right)$.
\end{prop}
\begin{proof} By Corollary 6.7 of \cite{PS11}, the module $((V_{\lambda,\mu})_{*})^*$ is an injective hull of $V_{\lambda,\mu}$ in $\Int_{\g^0}$. Moreover, we have the following chain of inclusions
    \[
\Coind_{\g^0}^{\g}(((V_{\lambda,\mu})_{*})^*)\supset \Coind_{\g^\geq}^{\g}(((V_{\lambda,\mu})_{*})^*)\supset \cT( ((V_{\lambda,\mu})_{*})^*)\supset \cT(V_{\lambda,\mu})\supset L^-_{\lambda,\mu}
    \]
where the latter inclusion follows from Theorem~\ref{thm:simples_are_tensor_modules}. Thus $L^-_{\lambda,\mu}\subset \Gamma \left( \Coind_{\g^0}^{\g}(((V_{\lambda,\mu})_{*})^*)\right)$.
\end{proof}

%%%%%%%%%%%%%%%%%%%%%%%%%%%%%%%%%
% References
%%%%%%%%%%%%%%%%%%%%%%%%%%%%%%%%%

%\bibliographystyle{alpha}
%\bibliography{global-weyl-modules-biblist}
%\bibliography{ref,outref,mathsci}

\end{document}